\documentclass[a4paper,11pt]{article}

\usepackage{hyperref} 
\usepackage[english]{babel}
\usepackage{amsfonts}
\usepackage[T1]{fontenc}
\usepackage[utf8]{inputenc}  
\usepackage{graphicx}  
\usepackage{amsmath}
\usepackage{array}
\usepackage{enumerate}                                                                                 
\usepackage{amscd}
\usepackage{stmaryrd}
\usepackage{amsthm}
\usepackage{mathabx}
\usepackage{pstricks}
\usepackage{dsfont}
\usepackage{lmodern}
\usepackage{epsfig}
\usepackage{pstricks-add}
\usepackage{fancyhdr}
\usepackage{geometry} 
\usepackage{pgf,tikz}
\usepackage{xcolor}
\usepackage{pgf,tikz}
\usepackage{mathrsfs}
\usepackage{amsfonts}
\usepackage{amssymb}
\usepackage{amstext}
\usepackage{dsfont}

\usetikzlibrary{arrows}
\usetikzlibrary{arrows}

\DeclareMathOperator{\R}{\mathbb{R}}

\DeclareMathOperator{\N}{\mathbb{N}}

\DeclareMathOperator{\cS}{\mathcal{S}}
\DeclareMathOperator{\cF}{\mathcal{F}}

\DeclareMathOperator{\cG}{\mathcal{G}}

\theoremstyle{plain}
\newtheorem{thm}{Theorem}[section] 

\theoremstyle{definition}

\newtheorem{prop}[thm]{Proposition}
\newtheorem{theo}[thm]{Theorem}
\newtheorem{defi}[thm]{Definition}
\newtheorem{rema}[thm]{Remark}
\newtheorem{lemm}[thm]{Lemma}
\newtheorem{coro}[thm]{Corollary}

\newtheorem{hypo}[thm]{Assumptions}

\begin{document}
\title{On the Long-Time Behaviour of Age and Trait Structured Population Dynamics }
\author{Tristan Roget}
\maketitle

\begin{abstract}
We study the long-time behaviour of a population structured by age and a phenotypic trait under a selection-mutation dynamics. By analysing spectral properties of a family of positive operators on measure spaces, we show the existence of eventually singular stationary  solutions. When the stationary measures are absolutely continuous with a continuous density, we show the convergence of the dynamics to the unique equilibrium. 
\end{abstract}

\maketitle

\section{Preliminaries and Main Results}
\subsection{Introduction}
Our ultimate goal is the understanding of the long-time behaviour  of a population where the individuals differ by their physical age $a\in\R_+$ and some hereditary variable $x\in\cS\subset \R^d$ called trait. The population evolves as follows. An individual with trait $x\in\cS$ and age $a\in\R_+$ has a death rate $D(x,a)+ cN$ where $D$ is the intrinsic death rate, $N$ is the population size and $c>0$ the competition rate. This individual gives birth at rate $B(x,a)$. At every birth, a mutation occurs with probability $p\in \left]0,1\right[$ and the trait of the newborn $y\in \cS$ is choosen according a distribution $k(x,a,y)dy$. Otherwise, the descendant inherits of the trait $x\in\cS$. In his thesis \cite{tran2006modeles}, Tran introduced an individual-based stochastic model to describe such a discrete population. The population is described by a random point measure
\begin{equation}\label{eq:introsto}
Z_t^K=\frac{1}{K}\sum_{i=1}^{N_t^K}\delta_{(x_i(t),a_{i}(t))}
\end{equation}
which evolves as a càdlàg Markov process with values in the set $\mathcal{M}^{+}(\cS\times\R_+)$ of positive finite measures on $\cS\times\R_+$ and each jump corresponds to birth or death of individuals. When the order $K$ of the size of the population goes to infinity such that $Z_0^K$ approximates a deterministic measure $n_0\in \mathcal{M}^{+}(\cS\times\R_+)$,
 it is shown (see \cite{tran2006modeles},\cite{tran2008large}) that the process approximates the unique weak solution $(n_t)_{t\geq 0}\in C(\R_+,\mathcal{M}^{+}(\cS\times\R_+))$ in the sense given by (\ref{eq:weakformpde})  of the partial differential equation
\begin{equation}\label{eq:edpintro}
\begin{cases}
\partial_t n_t(x,a)+\partial_a n_t(x,a)  =-\left(D(x,a)+ c\int_{\cS\times\R_+}n_t(y,\alpha)dyd\alpha\right)n_t(x,a),\\
n_t(x,0) = \cF\left[n_t\right](x),\quad (t,x,a)\in\R_+\times\cS\times\R_+,
\end{cases}
\end{equation}
where for all Borel subset $T$ of $\cS$
\begin{align}\label{eq:deff}
\cF&\left[n_t\right](T)\\
&=(1-p)\int_{T\times\R_+}B(x,\alpha)n_t(dx,d\alpha) +p\int_{T\times\cS\times\R_+}B(y,\alpha)k(y,\alpha,x)n_t(dy,d\alpha)dx.\nonumber
\end{align}
Recently, the well-posedness of measure solutions for a large class of partial differential equations including (\ref{eq:edpintro}) has also been established in \cite{canizo2013measure}, using a deterministic method. At our knowledge, nothing has been done about its long-time behaviour. In \cite{calsina2013steady},  the stationary problem is solved in $L^1(\cS\times\R_+)$ for a similar dynamics with a pure mutational kernel ($p=1$). The present paper is also motivated by \cite{bonnefon2015concentration}. The authors study the long-time behaviour of a selection-mutation dynamics with trait structure (and no age) and $p\in\left]0,1\right[$. They show the existence of stationary measures which can admit dirac masses in some traits and they analyse the long-time behaviour of the solutions when the stationary measure admits a bounded density. 
\\In this paper, we extend these facts to an age and trait structured population. 
We show the existence of non-trivial stationary measures for Equation (\ref{eq:edpintro}) (see Theorem \ref{maintheostatio}) which can be singular. When these measures are absolutely continuous with a continuous density, we show that the solutions of (\ref{eq:edpintro}) converge to the (unique) equilibrium (see Theorem \ref{maintheolongtime}). The method is based on the analysis of the linear dynamics. Indeed, the stationary states of (\ref{eq:edpintro}) are eigenvectors for the direct eigenvalue problem
\begin{equation}
\begin{cases}\label{eq:directeigenintro}
-\partial_a N(x,a) - (D(x,a)+\lambda)N(x,a) = 0 \\
N(x,0)=\cF\left[N\right](x).
\end{cases}
\end{equation}
The solutions of the dual problem 
\begin{equation}\label{eq:dualeigen}
\partial_a \phi(x,a)-(D(x,a)+\lambda)\phi(x,a)+\mathcal{G}\left[\phi\right](x,a)=0,
\end{equation}
where
\begin{equation}\label{eq:defg}
\cG\left[\phi\right](x,a)=B(x,a)\left((1-p)\phi(x,0) +p\int_{\cS}\phi(y,0)k(x,a,y)dy\right),
\end{equation}
give us some useful invariants and allow us to apply a method based on \cite{gwiazda2016generalized},\cite{perthame2006transport} leading to obtain an exponential rate of convergence for the linear dynamics to the stable distribution. 
\\As we will see, the study of the problem (\ref{eq:directeigenintro}) involves to understand spectral properties of a family of positive operators on the space of continuous functions on $\cS$ of the form
\begin{equation}\label{eq:opintrocoville}
(r+J)f(x)=r(x)f(x)+\int_{\cS}K(y,x)f(y)dy
\end{equation}
where $r$ is a continuous and positive function over $\cS$ and $K$ a continuous and non-negative kernel over $\cS$. In \cite{coville2010simple}, Coville finds a useful non-integrability criterion on the parameter $r$ which gives the existence of eigenfunctions associated with the principal eigenvalue of the operator $r+J$. When this criterion fails, he gives examples where there's no eigenfunction. Nonetheless, he shows in \cite{coville2013singular} that there are always principal eigenvectors in the of Radon measures space. Other properties of the operator are studied in \cite{coville2013pulsating}. In Section 1, we give a new, shorter and unified proof of all these results (see Theorem \ref{prop:eigenintegro}). Our approach is based on duality arguments (see Proposition \ref{prop:schaefermain} which is adapted from a result due to Krein-Rutman \cite{krein1948linear}) and allows us to obtain the existence of eigenvectors in a measures space. The criterion for the existence of principal eigenfunctions is also deduced. Our approach allows us to study at the same way the operator  $r+G$ defined by
\begin{equation}\label{eq:opdualintrocoville}
(r+G)f(x)=r(x)f(x)+\int_{\cS}K(x,y)f(y)dy
\end{equation}
which will be used for studying the dual problem (\ref{eq:dualeigen}).

\vspace{0.2cm}
In Section 1.3, we state our main results on the long-time behaviour of the solutions of (\ref{eq:edpintro}).
In Section 2, we study spectral properties of the operators of the form $r+J$, $r+G$ defined by (\ref{eq:opintrocoville}), (\ref{eq:opdualintrocoville}) and of their analogous operators in measure spaces. In Section 3, we apply these results to the study of the long-time behaviour of the linear dynamics. In Section 4, we deduce from the previous sections the proofs of our main results.
\vspace{0.2cm}
\paragraph{\textbf{Notations.}} Let $X$ be a metric space.
\begin{itemize}
\item[•]$C(X)$ (resp. $C^{+}(X)$) represents the sets of continuous functions from $X$ to $\R$ (resp. $\R_+$). $C_b(X)$ (resp. $C^{+}_b(X)$) represents the sets of continuous and bounded functions from $X$ to $\R$ (resp. $\R_+$). $\mathcal{M}(X)$ (resp. $\mathcal{M}^{+}(X)$) represents the set of finite Radon (resp. positive and finite Radon) measures on $X$. $\mathcal{M}_{loc}(X)$ (resp $\mathcal{M}^{+}_{loc}(X)$) represents the set of Radon (resp. positive Radon) measures on $X$. For any metric space $Y$, we denote by $C(X,Y)$ the set of continuous functions from $X$ to $Y$.
\item[•]We denote by $C_b^{1,0,1}=C_b^{1,0,1}(\R_+\times X\times\R_+)$ (resp. $C_c^{1,0,1}$) the set of continuous and bounded (resp. with compact support) functions from $\R_+\times X\times\R_+$ with continuous and bounded derivatives with respect to the first and third variables. We define similarly $C_b^{0,1}=C_b^{0,1}( X\times\R_+)$ and $C_c^{0,1}=C_c^{0,1}(X\times \R_+)$.
\item[•]For any $x\in X$ and $\epsilon\in \R_+^{*}$, we denote by $\mathcal{V}(x,\epsilon)$ (resp. $\overline{\mathcal{V}}(x,\epsilon)$) the open (resp. closed) ball centred in $x$ with radius $\epsilon$.
\end{itemize}
\subsection{Preliminaries}
We first give the main assumptions on the model. Then we recall some facts about topology of measure spaces and we conclude by giving some words about the well-posedness of the dynamics (\ref{eq:introsto}) and (\ref{eq:edpintro}).
\begin{hypo}\label{hypo1}$\text{}$
\begin{equation}\label{hyp:S}\tag{A1}
\cS = \overline{\Omega}, \quad\Omega\subset \R^d \text{ is open, bounded, connected  with Lipschitz boundary},
\end{equation}
\begin{equation*}\label{hyp:B}\tag{A2}
B,D\in C_b^{+}(\cS\times\R_+),\quad D(x,a)\geq \underline{D}>0,\quad k\in C_b^{+}(\cS\times\R_+\times\cS),
\end{equation*}
\begin{equation}\tag{A3}
p\in\left]0,1\right[,\quad  c>0
\end{equation}
and there exists $\epsilon_0>0$ and $I\subset \R_+$ with $\overset{\circ}{I}\neq 0$ such that for all $x\in\cS$ and $y\in \mathcal{V}(x,\epsilon_0)\cap \cS$:
\begin{equation}\label{eq:hypo6}\tag{A4}
I\subset \text{supp}(k(x,y,.))\cap \text{supp}(B(x,.)).
\end{equation}
\end{hypo}
\paragraph{Measure theory.} We recall some classical definitions and facts about topology on measures spaces. The Jordan decomposition theorem ensures that for any  $\mu\in \mathcal{M}(\cS\times\R_+)$, there is $\mu^{+},\mu^{-}\in\mathcal{M}^{+}(\cS\times\R_+)$ mutually singular, such that $\mu=\mu^{+}-\mu^{-}$. The total variation measure is defined by $\vert \mu\vert = \mu^{+}+\mu^{-}$ and the Total Variation norm by
\begin{equation*}
\Vert \mu\Vert_{\text{TV}}=\vert \mu\vert(\cS\times\R_+).
\end{equation*}
The Bounded Lipschitz norm is defined for any $\mu\in\mathcal{M}(\cS\times\R_+)$ by 
\begin{equation*}
\Vert \mu\Vert_{\text{BL}}=\sup\left\lbrace\left\vert\int_{\cS\times\R_+}f(x)\mu(dx)\right\vert: f\in W^{1,\infty}(\cS\times\R_+),\Vert f\Vert_{1,\infty}\leq 1\right\rbrace.
\end{equation*}
where $W^{1,\infty}(\cS\times\R_+)$ is the set of bounded Lipschitz functions from $\cS\times\R_+$ to $\R$ and $\Vert f\Vert_{1,\infty}=\Vert f\Vert_{\infty}+\text{Lip}(f)$ where $\text{Lip}(f)$ is the Lipschitz constant of $f$. We recall  (see \cite{villani2003topics}) that for any sequence $\mu_n\in \mathcal{M}^{+}(\cS\times\R_+)$ and $\mu\in\mathcal{M}^{+}(\cS\times\R_+)$, $\Vert \mu_n -\mu\Vert_{\text{BL}}\underset{n\rightarrow\infty}{\rightarrow} 0$ if and only if for all continuous and bounded function $f$  from $\cS\times\R_+$ to $\R$,
\begin{equation*}
\lim_{n\rightarrow\infty}\int_{\cS\times\R_+}f(x)\mu_n(dx)=\int_{\cS\times\R_+}f(x)\mu(dx),
\end{equation*}
i.e that $\mu_n\rightarrow \mu$ $\text{weakly}^{*}$ in $(C_b(\cS\times\R_+))'$. 
We denote by $C(\R_+,\mathcal{M}^{+}(\cS\times\R_+))$ the space of continuous maps from $\R_+$ to $\mathcal{M}^{+}(\cS\times\R_+)$ with respect to the Bounded Lipschitz norm. 
\paragraph{Well-posedness.} We precise the link between the stochastic process (\ref{eq:introsto}) and the partial differential equation (\ref{eq:edpintro}). We denote $\langle\mu,f\rangle =\int_{\cS\times\R_+} f(x,a)\mu(dx,da)$. Let us consider a sequence $(Z_0^K)_{K\geq 0}$ of $\mathcal{M}^{+}(\cS\times\R_+)$ valued random variables of the form 
\[Z_0^K=\frac{1}{K}\sum_{i=1}^{N_t^K}\delta_{(x_i,a_i)}.\] 
For each $K\in\mathbb{N}^{*}$, let $(Z_t^K)_{t\geq 0}$ be defined as the càdlàg measure-valued process started at $Z_0^K$ with infinitesimal generator $L^K$ given, for any $f\in C_b^{0,1}$ and $\mu\in\mathcal{M}^{+}(\cS\times\R_+)$ by
\begin{align}\label{eq:generateur}
L^K F_f&(\mu)=\int_{\cS\times\R_+}\partial_a f(x,a)F'(\langle \mu,f\rangle)\mu(dx,da)\\
&+K\int_{\cS\times\R_+}\left\lbrace(F(\langle \mu +\frac{\delta_{(x,0)}}{K},f\rangle)-F(\langle \mu,f\rangle))(1-p)B(x,a)\right.\nonumber\\
&\left.+\left(\int_{\cS}(F(\langle \mu +\frac{\delta_{(y,0)}}{K},f\rangle)-F(\langle \mu,f\rangle))B(x,a)pk(x,a,y)dy\right)\right.\nonumber\\
&\left. +(F(\langle \mu -\frac{\delta_{(x,a)}}{K},f\rangle)-F(\langle \mu,f\rangle))\left(D(x,a)+c\langle\mu,1\rangle\right)\right\rbrace\mu(dx,da)\nonumber
\end{align}
where $F_f(\mu):=F(\langle \mu,f\rangle)$ (we note that the set of functions of the form $F_f$ is sufficient to characterise the infinitesimal generator, as it is proved in \cite{dawson1993measure}).
The following proposition allows to obtain the solutions of (\ref{eq:edpintro}) as a large population limit of the stochastic process $Z^K$. We refer to \cite{tran2006modeles} for the proof.
\begin{prop} \label{prop:largepop}
Assume Assumptions \ref{hypo1} and that $Z_0^K$ converges in law to $n_0\in\mathcal{M}^{+}(\cS\times\R_+)$ as $K\rightarrow\infty$. Then, the sequence of processes $(Z^K)_{K\geq 0}$ converges in law (on finite time interval) to the unique weak solution $(n_t)_{t\geq 0}\in C(\R_+,\mathcal{M}^{+}(\cS\times\R_+))$ of (\ref{eq:edpintro}) which satisfies for all $f\in  C_b^{1,0,1}$ and $t\in\R_+$,
\begin{align}\label{eq:weakformpde}
\int_{\cS\times\R_+}f&(t,x,a) n_t(dx,da) \nonumber\\
& =\int_{\cS\times\R_+} f(0,x,a)n_0(dx,da)+\int_{0}^{t}\int_{\cS\times\R_+}\left(\partial_s f(s,x,a)+\partial_a f(s,x,a)\right.\nonumber\\
&\left.-\left(D(x,a)+c n_s(\cS\times\R_+) \right)f(s,x,a)+\mathcal{G}\left[f(s,.)\right](x,a)\right)n_s(dx,da)ds
\end{align}
where $\mathcal{G}$ has been defined in (\ref{eq:defg}).
\end{prop}
\subsection{Main Results}
Let us introduce some notations. For any $(\lambda,x,y,a)\in \left]-\underline{D},+\infty\right[\times\cS^2\times\R_+$, we define:
\begin{equation}\label{defifonction}
\begin{cases}
R_{\lambda}(x,a)=\exp\left(-\int_{0}^{a}D(x,\alpha)d\alpha)-\lambda a\right),\\
r_{\lambda}(x)=(1-p)\int_{\R_+}B(x,a)R_{\lambda}(x,a)da,\\
K_{\lambda}(x,y)=p\int_{\R_+}B(x,a)k(x,a,y)R_{\lambda}(x,a)da.
\end{cases}
\end{equation}
For any $\lambda\in\left]-\underline{D},+\infty\right[$, we define the linear operator $\tilde{r}_\lambda + \tilde{J}_{\lambda}:\mathcal{M}(\cS)\rightarrow\mathcal{M}(\cS)$ by  
\begin{equation}\label{eq:defopmeasure}
(\tilde{r}_\lambda+\tilde{J}_\lambda)\mu =r_\lambda(x)\mu + \left(\int_{\cS}K_\lambda(y,x)\mu(dy)\right)dx
\end{equation}
and we denote by $\rho(\tilde{r}_\lambda+\tilde{J}_\lambda)$ its spectral radius.
We now give the main results of the paper. The first one shows the existence of stationary states for the dynamics (\ref{eq:weakformpde}), under some assumption on the spectral radius of the operators introduced above (similarly as in \cite{calsina2013steady}). This assumption is related to the supercriticality of the associated linear dynamics. 
\begin{theo}\label{maintheostatio}
Assume Assumptions \ref{hypo1} and $\rho(\tilde{r}_0 +\tilde{J}_0)>1$. There exists a non-zero solution $\overline{n}\in \mathcal{M}^{+}(\cS\times\R_+)$ of:
\begin{equation}\label{eq:statioproblem}
\forall f\in C_b^{0,1},\quad \int_{\cS\times\R_+}\left(\partial_a f - \left(D+c\int_{\cS\times\R_+}\overline{n}\right)f+\mathcal{G}\left[f\right]\right)(x,a)\overline{n}(dx,da)=0,
\end{equation}
which is given by
\begin{equation*}
\overline{n}(x,a)=\mu_{\lambda^{*}}(dx)R_{\lambda^{*}}(x,a)da
\end{equation*}
where $\lambda^{*}>0$ is solution of the equation $\rho(\tilde{r}_{\lambda^{*}}+\tilde{J}_{\lambda^{*}})=1$ and $\mu_{\lambda^{*}}\in \mathcal{M}^{+}(\cS)$ is an eigenvector of $\tilde{r}_{\lambda^{*}}+\tilde{J}_{\lambda^{*}}$ associated with the eigenvalue $\rho(\tilde{r}_{\lambda^{*}}+\tilde{J}_{\lambda^{*}})=1$. 
\end{theo}
Let us now introduce an additional regularity assumption which allows us to obtain the continuity of the solutions with respect to the initial conditions (see Lemma \ref{lemm:canizo}).
\begin{hypo}\label{hypo2}$B,D\in W^{1,\infty}(\cS\times\R_+)$ and $k\in W^{1,\infty}(\R_+\times\cS\times\R_+)$.
\end{hypo}
We now focus on the case where there exists a stationary measure $\overline{n}$ which admits a continuous density (we keep the same notation $\overline{n}$ for the density).
\begin{theo}\label{maintheolongtime}
Assume Assumptions \ref{hypo1}, \ref{hypo2}. Assume that $\rho(\tilde{r}_0 +\tilde{J}_0)>1$ and that there exists a solution  $\overline{n}\in C(\cS, L^1(\R_+))$ of (\ref{eq:statioproblem}). Assume that there exist $\underline{B},\underline{k}>0$ such that $B\geq \underline{B}$ and $k\geq\underline{k}$. Let $(n_t)_{t\geq 0}\in C(\R_+, \mathcal{M}^{+}(\cS\times\R_+))$ be the solution of (\ref{eq:weakformpde}) started at $n_0\in \mathcal{M}^{+}(\cS\times\R_+)\setminus \lbrace 0\rbrace$. Then we have
\begin{equation*}
\lim_{t\rightarrow\infty}\Vert n_t -\overline{n}\Vert_{\text{TV}}=0
\end{equation*}
and $\overline{n}$ is the unique stationary measure.
\end{theo}
\begin{rema}
Let us give a sufficient condition to obtain $\overline{n}\in C(\cS,L^1(\R_+))$. By Theorem \ref{maintheostatio} the trait marginal $\mu_{\lambda^{*}}(dx)$ of $\overline{n}$ is an eigenvector of $\tilde{r}_{\lambda^{*}}+\tilde{J}_{\lambda^{*}}$ associated with the eigenvalue $\rho(\tilde{r}_{\lambda^{*}}+\tilde{J}_{\lambda^{*}})=1$. By Lemma \ref{prop:eigenintegroparam} (iii) and (ii), we obtain that if $\text{Leb}(\Sigma_{\lambda^{*}})>0$ or if $\text{Leb}(\Sigma_{\lambda^{*}})=0$ and $\frac{1}{\overline{r}_{\lambda^{*}}-r_{\lambda^{*}}}\notin L^1(\cS)$ the measure $\mu_{\lambda^{*}}(dx)$ is absolutely continuous with a continuous density, and that $\overline{n}\in C(\cS,L^1(\R_+))$.
\end{rema}
\begin{rema}
If $\overline{n}$ admits a non-trivial singular part, Theorem 1.5 is false since any $L^1(\cS\times\R_+)$-solution of (\ref{eq:weakformpde}) can't converge to a singular measure in total variation norm. 
\end{rema}
\section{Spectral Properties of some Positive Operators}
\subsection{Position of the Problem and Results}
Let us consider a subset $\cS$ of $\R^d$ which satisfies Assumption (\ref{hyp:S}).
We analyse the spectral properties of the operators $r+J,r+G:C(\cS)\rightarrow C(\cS)$ defined respectively by (\ref{eq:opintrocoville}), (\ref{eq:opdualintrocoville}) and of their analogous operators on measure spaces $\tilde{r}+\tilde{J},\tilde{r}+\tilde{G}:\mathcal{M}(\cS)\rightarrow\mathcal{M}(\cS)$ defined similarly by
\begin{align*}
(\tilde{r}+\tilde{J})\mu =r(x)\mu + \left(\int_{\cS}K(y,x)\mu(dy)\right)dx,\\
(\tilde{r}+\tilde{G})\mu =r(x)\mu + \left(\int_{\cS}K(x,y)\mu(dy)\right)dx.
\end{align*}
\begin{hypo}\label{hypo4}
\begin{equation}\label{eq:hypo1}\tag{A5}
r\in C(\cS)\text{ is positive},
\end{equation}
\begin{equation}\label{eq:hypo2}\tag{A6}
K\in C(\cS\times\cS)\text{ is non-negative, }
\end{equation}
\begin{equation}\label{eq:hypo3}\tag{A7}
\exists \epsilon_0,c_0>0,\quad\inf_{x\in\cS}\left(\inf_{y\in \mathcal{V}(x,\epsilon_0)\cap \cS} K(x,y)\right)>c_0.
\end{equation}
\end{hypo}
\begin{rema}
Assume Assumption (\ref{eq:hypo3}), then we have
\begin{equation}
\inf_{x\in\cS}\left(\inf_{y\in\mathcal{V}(x,\epsilon_0)\cap \cS}K(y,x)\right)>c_0.
\end{equation}
Indeed, let $x\in\cS$ and $y\in\mathcal{V}(x,\epsilon_0)\cap \cS$. Then $x\in\mathcal{V}(y,\epsilon_0)$ and $K(y,x)>c_0$.
\end{rema}
The following result deals with the spectral properties of the operators introduced above. We denote by $\rho(r+J)$ the spectral radius of the operator $r+J$ (and similarly for $r+G$). The reader can refer to Appendix A for terminology and recalls about spectral theory. The following theorem is proved in Section 2.2.
\begin{theo}\label{prop:eigenintegro}
Assume Assumptions \ref{hypo4}. There exists $\mu\in \mathcal{M}^{+}(\cS)$ such that
\begin{equation*}
\rho(r+G)\mu =(\tilde{r}+\tilde{J})\mu
\end{equation*}
which satisfies $\mu(A)>0$ for any Borel subset $A$ of $\cS$ such that $\text{Leb}(A)>0$ ($\text{Leb}$ denotes the Lebesgue measure on $\cS$). Moreover, let us denote $\overline{r}=\sup_{x\in\cS}r(x)$ and $\Sigma=\lbrace x\in\cS:r(x)=\overline{r}\rbrace$. Then we have:
\begin{itemize}
\item[(i)]$\overline{r}\leq \rho(r+G)$.
\item[(ii)]$\overline{r}<\rho(r+G)$ if and only if there exists $u\in C(\cS)$, $u>0$ such that $\mu =u(x)dx$. In this case, $\rho(r+G)$ is an eigenvalue of $r+J$ with algebraic multiplicity equals to one, associated with the eigenfunction $u$.
\item[(iii)]If $\text{Leb}(\Sigma)> 0$ or if, $\text{Leb}(\Sigma)=0$ and $\frac{1}{\overline{r}-r}\notin L^1(\cS)$, we have $\overline{r}<\rho(r+G)$.
\item[(iv)]If $\rho(r+G)=\overline{r}$, we have $\mu=\mu^s +h(x)dx$ with $h\in L^1(\cS)$ and either $\mu^s=0$, or $\mu^s\neq 0$ and $\text{supp}(\mu^s)\subset \Sigma$.
\item[(v)]$\rho(r+J)=\rho(r+G)=\rho(\tilde{r}+\tilde{J})=\rho(\tilde{r}+\tilde{G})$.
\end{itemize}
Moreover, the same results are true exchanging $J$ and $G$, $\tilde{J}$ and $\tilde{G}$.
\end{theo}
\begin{proof}
See Section 2.2.
\end{proof}
\begin{rema}
In \cite{coville2010simple},\cite{coville2013singular}, Coville studies some spectral properties of the operators introduced above. To do so, he introduces the generalised principal eigenvalue 
\begin{equation*}
\lambda_p(r+J)=\sup\lbrace \lambda\in\R\vert \exists \varphi\in C(\cS), \varphi > 0 \text{ such that }(r+J)\varphi +\lambda\varphi \leq 0\rbrace 
\end{equation*}
which generalises the Perron-Frobenius eigenvalue for irreducible matrices with non-negative coefficients.  The point (iii) is similar to the criterion obtained (in a more general setting) in \cite{coville2010simple}. The point (iv) is contained in the results of \cite{coville2013singular}.  The point (v) is new. It is crucial for the proof of Propositions \ref{theo:main} and \ref{theo:maindual}.
\end{rema}
\subsection{Proof of Theorem \ref{prop:eigenintegro}}
The proof of Theorem \ref{prop:eigenintegro} is given at the end of this section. We start by proving a lemma in which some well known facts about the operators introduced previously are recalled. We give a proof for the convenience of the reader. We denote by $\rho_e(r+J)$ the essential spectral radius of $r+J$ (see Appendix A).
\begin{lemm}\label{lemm:propop}
Assume Assumptions \ref{hypo4}. 
\begin{itemize}
\item[(i)] The operators $r+J$ and $r+G$ are positive and irreducible on $C(\cS)$.
\item[(ii)] The operators $J$ and $G$ are compact on $C(\cS)$.
\item[(iii)]$\overline{r}=\rho_e(r+J)=\rho_e(r+G)$.
\end{itemize}
\end{lemm}
\begin{proof}
(i): Since $r$ is positive and $K$ is non-negative, $r+J$ and $r+G$ are positive endomorphisms of $C(\cS)$. To prove irreducibility, it suffices to prove that there is $m\in\N^{*}$ such that for all $f\in C^{+}(\cS)$ and $x\in \cS$, $(r+J)^m f(x)>0$ (see Definition \ref{defi:irred} and Lemma \ref{lemm:quasi}). Since the set $\cS$ is compact, there exist $n\in\mathbb{N}^{*}$ and $(B_i)_{i=1}^{n}$ a family of balls  with radius $\epsilon_0/4$ such that $\cS\subset\cup_{i=1}^n B_i$. Let $f\in C^{+}(\cS)$ be non-zero and let $I$ be an open subset of $\cS$ such that $f$ is positive on $I$. For all $x\in\cS$ we have
\begin{align*}
(r+J)^n f(x)&\geq \int_{I}dx_1 f(x_1)K(x_1,x_2)\int_{\cS}dx_2\ldots\int_{\cS}dx_n K(x_n,x)\\
&\geq C\int_{I\cap B_{i_1}\cap\cS}dx_1 K(x_1,x_2)\int_{B_{i_2}\cap\cS}dx_2\ldots\int_{B_{i_{n}}\cap\cS}dx_n K(x_n,x)
\end{align*}
where $C>0$ and $(i_1,\ldots,i_n)\in \llbracket 1,n\rrbracket^n$ satisfies: $B_{i_1}\cap I$ has non-empty interior; for any $k\in\llbracket 1,n-1\rrbracket $, $u\in B_{i_{k}}\cap\cS$, $v\in B_{i_{k+1}}\cap\cS$, $K(u,v)>c_0$ and for any $u\in B_n\cap \cS$, $K(u,x)>c_0$. It comes that
\begin{equation*}
(r+J)^n f(x)>0
\end{equation*}
and $r+J$ is irreducible. The proof is similar for $r+G$. 
\\(ii): Let $M$ be a bounded subset of $C(\cS)$ and let $C_M$ be a positive constant such that for all $f\in M$, $\Vert f\Vert_{\infty}<C_M$. Applying Ascoli's criterion, we prove that $J(M)$ is relatively compact in $C(\cS)$. Let $x\in \cS$ and $f\in M$, we have
\begin{align*}
\vert Jf(x)\vert &\leq \Vert f\Vert_{\infty}\sup_{z\in\cS}\int_{\cS}K(y,z)dy\\
&\leq C_M \sup_{z\in\cS}\int_{\cS}K(y,z)dy.
\end{align*}
Then the set $\lbrace Jf(x),f\in M\rbrace$ is bounded and so relatively compact in $\R$. We check the equi-continuity condition. Let $\epsilon>0$, since $K$ is uniformly continuous on $\cS\times\cS$, there exists $\delta > 0$ such that if  $\Vert x_1-x_2\Vert +\Vert y_1 -y_2\Vert <\delta$, we have $\vert K(x_1,y_1)-K(x_2,y_2)\vert <\frac{\epsilon}{C_M \text{Leb}(\cS)} $. Let $y\in\cS$ such that $\Vert x- y\Vert<\delta$. For all $f\in M$, we have
\begin{align*}
\vert Jf(y) - Jf(x)\vert &=\left\vert \int_{\cS}f(z)(K(z,y)-K(z,x))dz\right\vert\\
&\leq C_M \int_{\cS}\vert K(z,y)-K(z,x)\vert dz <\epsilon
\end{align*}
which allows us to conclude for the compactness. The proof is similar for $G$.
\\(iii): Let us note from Lemma \ref{lemm:spectremult} that the essential spectrum of $r$ is $\lbrace r(x), x\in\cS\rbrace$. Moreover $J$ is compact. We deduce that $\rho_e(r+J)=\rho_e(r)$ and that $\overline{r}=\rho_e(r+J)$. The proof is similar for $\rho_e(r+G)$.
\end{proof}
The following lemma makes a duality link between the operators introduced above. It is crucial for the proof of Theorem \ref{prop:eigenintegro}. 
\begin{lemm}\label{lemm:adjoint}
Assume Assumptions \ref{hypo4}. We have
\begin{equation}
(r+J)'=\tilde{r}+\tilde{G}\quad (r+G)'=\tilde{r}+\tilde{J}
\end{equation}
where $(r+J)'$ is the adjoint operator of $r+J$ and $(r+G)'$ is defined similarly.
\end{lemm}
\begin{proof}
Let $f\in C(\cS)$ and $\mu\in\mathcal{M}(\cS)$,
\begin{align*}
\int_{\cS}\mu(dx)(r+J)f(x)&=\int_{\cS}\mu(dx)\left(r(x)f(x)+\int_{\cS}K(y,x)f(y)dy\right)\\
&=\int_{\cS}f(x)r(x)\mu(dx)+\int_{\cS}\left(f(y)\int_{\cS}K(y,x)\mu(dx)\right)dy\\
&=\int_{\cS}f(x)\left(r(x)\mu(dx)+dx\int_{\cS}K(x,y)\mu(dy)\right)\\
&=\int_{\cS}f(x)(\tilde{r}+\tilde{G})\mu(dx)
\end{align*}
where we used Fubini's Theorem. The proof is similar for $(r+G)'$.
\end{proof}
The next result is easily adapted from \cite[ Appendix §2.2.6]{schaefer1971graduate} (it was originally introduced by Krein-Rutman \cite{krein1948linear}) and \cite[Appendix §3.3.3]{schaefer1971graduate}. Combined  with Lemma \ref{lemm:adjoint}, it is the main tool for the proof of Theorem \ref{prop:eigenintegro}.
\begin{prop}\label{prop:schaefermain}
\begin{itemize}
\item[(i)]Let $T$ be a positive endomorphism of $C(\cS)$. The spectral radius $\rho(T)$ is an eigenvalue of $T'$ associated with an eigenvector which belongs to $\mathcal{M}^{+}(\cS)$.
\item[(ii)]Let $T$ be an irreducible endomorphism of $C(\cS)$. Then the spectral radius $\rho(T)$ is the only possible eigenvalue of $T$ associated with a non-negative eigenvector. Moreover, if $\rho(T)$ is a pole of the resolvent, it is an eigenvalue of $T$ with algebraic multiplicity equals to one.
\end{itemize}
\end{prop}
We give now a technical lemma.
\begin{lemm}\label{lemm:approxfonc}
Assume Assumptions \ref{hypo4}. Let $x_0\in\Sigma=\lbrace x\in\cS:r(x)=\overline{r}\rbrace$. There exists a family $(r_j)_{j\geq 0}$ of $C^{+}(\cS)$ which satisfy for all $j\geq 0$:
\begin{itemize}
\item[(i)]For all $x\in\cS$, $r_j(x)\geq r_{j+1}(x)$,
\item[(ii)] $r_j(x_0)=\overline{r}_j=\overline{r}$ and $\text{Leb}(\Sigma_j)>0,$ where $\Sigma_j=\lbrace x\in\cS:r_j(x)=\overline{r}_j\rbrace$,
\item[(iii)]$\Vert r_j -r\Vert_{\infty} \underset{j\rightarrow\infty}{\rightarrow}0$.
\end{itemize}
\end{lemm}
\begin{proof}
Let $x_0\in\Sigma$ be fixed. For all $\epsilon > 0$ sufficiently small, we define the closed set $A_{\epsilon}=\left(\mathcal{V}(x_0,\epsilon)^c \cup\overline{\mathcal{V}}(x_0,\epsilon/2)\right)\cap \cS$ and a map $g_{\epsilon}\in C(A_{\epsilon})$ by $g_{\epsilon}(x)=r(x)$ if $x\in \mathcal{V}(x_0,\epsilon)^c$ and $g_{\epsilon}(x)=\overline{r}$ if $x\in\overline{\mathcal{V}}(x_0,\epsilon/2)$. By Tietze Theorem, we extend $g_\epsilon$ in a continuous function $h_{\epsilon}$ on $\cS$ such that $\Vert h_{\epsilon}\Vert_{\infty}=\Vert g_{\epsilon}\Vert_{\infty}$
We introduce $r_{\epsilon}\in C(\cS)$ defined by $r_\epsilon(x)=\max( h_{\epsilon}(x),r(x))$. It is straightforward to check that: 1) $\lim_{\epsilon\rightarrow 0}\Vert r_{\epsilon} -r\Vert_{\infty}=0$; $\text{Leb}(\Sigma_{\epsilon})>0$; 3) $r_\epsilon(x)\geq r(x) $ and $\sup_{x\in\cS}r_{\epsilon}(x) = \overline{r}$. We conclude by proving that we can extract a decreasing subsequence of the family $r_{\epsilon}$ which converges uniformly to $r$. To do so, we fix $\epsilon_0>$ small and we define a sequence $(\epsilon_k)_{k\geq 0}$ by $\epsilon_{k+1}=\frac{\epsilon_{k}}{2}$. We check that the sequence $(r_{\epsilon_k})_{k\geq 0}$ is decreasing. Indeed, let $k\geq 0$. If $x\in \overline{\mathcal{V}}(x_0,\epsilon_{k+1}/2)$, then $x\in \overline{\mathcal{V}}(x_0,\epsilon_{k}/4)$ and $r_{\epsilon_{k+1}}(x)=h_{\epsilon_{k+1}}(x)=\overline{r}=h_{\epsilon_{k}}(x)=r_{\epsilon_{k}}(x)$. If $x\in A_{\epsilon_{k+1}}^c$ we have $\epsilon_{k+1}/2<\Vert x- x_0\Vert<\epsilon_{k+1}=\epsilon_{k}/2$. So we have $r_{\epsilon_{k+1}}(x)\leq\overline{r}=h_{\epsilon_{k}}(x)=r_{\epsilon_{k}}(x)$. Finally, if $x\in \mathcal{V}(x_0,\epsilon_{k+1})^c$, $r_{\epsilon_{k+1}}(x)=r(x)\leq r_{\epsilon_{k}}(x)$. So we have proved that for all $x\in\cS$, $r_{\epsilon_{k+1}}(x)\leq r_{\epsilon_k}(x)$.
\end{proof}
Let us now prove Theorem \ref{prop:eigenintegro}. 
\begin{proof}[Proof of Theorem \ref{prop:eigenintegro}]
By Lemma \ref{lemm:adjoint} and Proposition \ref{prop:schaefermain} (i) applied to the endomorphism $r+G$, there exists a non-zero measure $\mu\in\mathcal{M}^{+}(\cS)$ such that for all Borel and bounded functions $f:\cS\rightarrow \R$,
\begin{equation}\label{proof:measure}
\int_{\cS}f(x)\left(\int_{\cS}K(y,x)\mu(dy)\right)dx+\int_{\cS}f(x)(r(x)-\rho(r+G))\mu(dx)=0.
\end{equation}
Assume that there exists a largest Borel subset $A$ of $\cS$, $A\neq \cS$ such that $\text{Leb}(A)>0$ and $\mu(A)=0$. Choosing $f=\mathds{1}_{A}$ in (\ref{proof:measure}), we deduce that for all $x\in A$,
\begin{equation}\label{proo:integrozero}
\int_{\cS}K(y,x)\mu(dy)=0.
\end{equation}
Let $x_0\in A$ be such that $\mathcal{V}(x_0,\epsilon_0)\cap\cS\not\subset A$. By Assumption (\ref{eq:hypo3}) and (\ref{proo:integrozero}) we obtain
\begin{equation*}
0\geq \int_{\mathcal{V}(x_0,\epsilon_0)\cap \cS}K(y,x_0)\mu(dy)\geq c_0\mu(\mathcal{V}(x_0,\epsilon_0)\cap \cS)
\end{equation*}
and $\mu(A\cup (\mathcal{V}(x_0,\epsilon_0)\cap\cS))=0$ which is absurd by definition of $A$. Since $\mu(\cS)>0$, we conclude that for all Borel subset $A$ of $\cS$ such that $\text{Leb(A)}>0$, we have $\mu(A)>0$. 
\\(i): By Lemma \ref{lemm:propop} (iii), we have $\rho_e(r+G)=\rho_e(r)=\overline{r}$. It comes that $\overline{r}\leq\rho(r+G)$. 
\\(ii): Assume that $\overline{r}<\rho(r+G)$. Let $f\in C(\cS)$. Then the map $x\in \cS \longmapsto \frac{f(x)}{\rho(r+G)-r(x)}$ is continuous and bounded. We get
\begin{align*}
\int_{\cS}f(x)\mu(dx)&=\int_{\cS}\frac{f(x)}{\rho(r+G)-r(x)}(\rho(r+G)-r(x))\mu(dx)\\
&=\int_{\cS}\frac{f(x)}{\rho(r+G)-r(x)}\left(\int_{\cS}K(y,x)\mu(dy)\right)dx\\
&=\int_{\cS}f(x)\frac{\int_{\cS}K(y,x)\mu(dy)}{\rho(r+G)-r(x)}dx.
\end{align*} 
So we have $\mu=u(x)dx$ with 
\[u(x)=\frac{\int_{\cS} K(y,x)\mu(dy)}{\rho(r+G)-r(x)}\]
a continuous, non-negative function on $\cS$. Therefore, for all $f\in C(\cS)$
\begin{equation*}
\int_{\cS}f(x)\left(u(x)(r(x)-\rho(r+G))+\int_{\cS}K(y,x)u(y)dy\right)dx=0
\end{equation*}
and we deduce that for all $x\in\cS$, $(r+J)u(x)=\rho(r+G)u(x)$. Assume that there exists $x_0\in\cS$ such that $u(x_0)=0$. Then we have
\begin{equation*}
c_0\int_{\mathcal{V}(x_0,\epsilon_0)\cap\cS}u(y)dy\leq \int_{\cS}K(y,x_0)u(y)dy=0
\end{equation*}
which is absurd by the first statement we proved. Then, $u$ is positive on $\cS$.  Since $\rho(r+G)$ is an eigenvalue of $r+J$ associated with a positive eigenfunction, Proposition \ref{prop:schaefermain} (ii) gives that $\rho(r+G)=\rho(r+J)$. We deduce that $\rho(r+J)>\overline{r}=\rho_e(r+J)$. It comes from Proposition \ref{prop:browder} in Appendix that $\rho(r+J)$ is a pole of the resolvent. Since $r+J$ is irreducible (see Lemma \ref{lemm:propop} (i)), it comes from Proposition \ref{prop:schaefermain} (ii) that the algebraic multiplicity of $\rho(r+J)$ is equals to one. Conversely, assume now that there exists $u\in C^{+}(\cS)$ such that $\mu=u(x)dx$. Then since for all $x\in\cS$
\begin{equation*}
u(x)=\frac{1}{\rho(r+G)-r(x)}\int_{\cS}K(y,x)u(y)dy
\end{equation*}
we deduce that $\overline{r}<\rho(r+G)$.
\\(iii): Assume first that $\text{Leb}(\Sigma)>0$. Choosing $f=\mathds{1}_{\Sigma}$ in (\ref{proof:measure}) and by the definition of $\Sigma$, we get that
\begin{equation*}
\int_{\Sigma}\left(\int_{\cS}K(y,x)\mu(dy)\right)dx =(\rho(r+G)-\overline{r})\mu(\Sigma).
\end{equation*}
Since $\text{Leb}(\Sigma)>0$, we know that $\mu(\Sigma)>0$. Since we have $\inf_{x\in\cS}\int_{\cS}K(y,x)\mu(dy)>0$, we deduce that $\overline{r}<\rho(r+G)$.
Assume now that $\text{Leb}(\Sigma)=0$ and $\frac{1}{\overline{r}-r}\notin L^1(\cS)$. Here, our calculations are inspired by \cite{coville2010simple}. Let 
\[A:=\frac{\mu(\cS)}{\inf_{x\in\cS}\int_{\cS}K(y,x)\mu(dy)}>0\]
and $B>A$. There exists $F$ a closed subset of $\Sigma^{c}$ such that
\begin{equation*}
B<\int_{F}\frac{1}{\overline{r}-r(x)}dx<+\infty.
\end{equation*}
Then, the map \hspace{0.05cm} $\epsilon\in\left[0,+\infty\right[\mapsto \int_{F}\frac{1}{\overline{r}-r(x)+\epsilon}dx$ \hspace{0.05cm} is continuous and strictly decreasing. So, there exists $\epsilon_0>0$ such that 
\begin{equation*}
A<\int_{F}\frac{1}{\overline{r}-r(x)+\epsilon_0}dx.
\end{equation*}
Choosing $f(x)=\mathds{1}_{F}(x)\frac{1}{\rho(r+G)-r(x)}$ in (\ref{proof:measure}), we have
\begin{equation*}
\int_{F}\frac{1}{\rho(r+G)-r(x)}\left(\int_{\cS}K(y,x)\mu(dy)\right)dx=\int_{F}\mu(dx)\leq \mu(\cS)
\end{equation*}
and
\begin{equation*}
\int_{F}\frac{1}{\rho(r+G)-\overline{r}+\overline{r}-r(x)}dx\leq A.
\end{equation*}
Since the map \hspace{0.05cm} $\epsilon\in\left[0,+\infty\right)\mapsto \int_{F}\frac{1}{\overline{r}-r(x)+\epsilon}dx$\hspace{0.05cm} is strictly decreasing, it comes that $\epsilon_0<\rho(r+G)-\overline{r}$  and $\overline{r}<\rho(r+G)$. 
\\(iv): Assume that $\rho(r+G)=\overline{r}$. Let $\mu=\mu^s + h(x)dx$ be the Lebesgue decomposition of the measure $\mu$ with $h\in L^1(\cS)$ and $\mu^s$ the singular part of the measure $\mu$, i.e there exists $E$ a measurable subset of $\cS$ such that $\text{Leb}(\cS)=\text{Leb(E)}$ and $\mu^s(E^c)=\mu^s(\cS)$. It comes from (\ref{proof:measure}) with $f=\mathds{1}_{E^c}$ that
\begin{equation*}
\int_{\cS}\mu^s(dx)(r(x)-\rho(r+G))=\int_{E^c}\mu^s(dx)(r(x)-\rho(r+G))=0.
\end{equation*}
Assume that $\mu^s\neq 0$, then we deduce that the support of the measure $\mu^s$ is a subset of $\Sigma$. 
\\(v): Assume first that $\frac{1}{\overline{r}-r}\notin L^1(\cS)$ and $\text{Leb}(\Sigma)=0$, or $\text{Leb}(\Sigma)>0$. By (iii) we get that $\overline{r}<\rho(r+G)$. By (ii), we deduce that there exists $u\in C(\cS)$, $u>0$ such that $(r+J)u=\rho(r+G)u$. By Proposition \ref{prop:schaefermain} (ii), we have $\rho(r+J)=\rho(r+G)$. Assume now that $\frac{1}{\overline{r}-r}\in L^1(\cS)$. Let $x_0\in\Sigma$. By Lemma \ref{lemm:approxfonc}, there is a sequence $(r_j)_{j\geq 0}$ of $C^{+}(\cS)$ which satisfies: 1) $\Vert r_{j} -r\Vert_{\infty}\rightarrow 0$; 2) $\text{Leb}(\Sigma_j)>0$ and $\overline{r}_j=\overline{r}$ ($\Sigma_j=\lbrace x\in\cS: r_j(x)=\overline{r}_{j}\rbrace$); 3) for all $x\in\cS$, $r_{j+1}(x)\leq r_{j}(x)$. By the first part of the proof of (v), we deduce that $\rho(r_{j}+J)=\rho(r_{j}+G)$ and we conclude that $\rho(r+J)=\rho(r+G)$ taking the limit $j\rightarrow\infty$, using the monotonicity (see Proposition \ref{annex:montoto} (i) in Appendix A) and the upper semi-continuity of the spectral radius (see Lemma \ref{lemm:uppersemi} in Appendix A). The others equalities are proved arguing that $(r+J)'=\tilde{r}+\tilde{G}$ and $(r+G)'=\tilde{r}+\tilde{J}$.
\end{proof}
\section{The Linear Dynamics}
In this section, we apply the results of the previous section to analyse the long-time behaviour of the solutions $(v_t)_{t\geq 0}\in C(\R_+,\mathcal{M}^{+}(\cS\times\R_+))$ of the linear equation:
\begin{equation}\label{eq:edplin}
\begin{cases}
\partial_t v_t(x,a)+\partial_a v_t(x,a)  =-D(x,a)v_t(x,a),\quad (t,x,a)\in \R_+\times\cS\times\R_+,\\
v_t(x,0) = \cF\left[v_t\right](x),\quad v_0\in \mathcal{M}^{+}(\cS\times\R_+).
\end{cases}
\end{equation}
The well posedness of solutions $(v_t)_{t\geq 0}\in C(\R_+,\mathcal{M}^{+}(\cS\times\R_+))$ is proved in \cite{tran2006modeles} using the microscopic approach, and in \cite{canizo2013measure} using a deterministic method. 
We start by proving that Assumptions \ref{hypo1} imply Assumptions \ref{hypo4} for $r_\lambda$ and $K_\lambda$  defined in (\ref{defifonction}).
\begin{lemm}\label{lemm:tech1}
Assume Assumptions \ref{hypo1}.
\begin{itemize}
\item[1)]For all $\lambda\in\left]-\underline{D},+\infty\right[$, the maps $r_{\lambda}$ and $K_{\lambda}$ are well-defined, continuous, respectively positive and non-negative. There exist $\epsilon_0,c_0>0$ such that 
\begin{equation*}
\inf_{x\in\cS}\left(\inf_{y\in \mathcal{V}(x,\epsilon_0)\cap \cS}K_{\lambda}(x,y)\right)>c_0.
\end{equation*}
\item[2)]Moreover we have:
\begin{itemize}
\item[(i)]For all $-\underline{D}<\lambda_1<\lambda_2$ and $(x,y)\in\cS\times\cS$, $r_{\lambda_1}(x)>r_{\lambda_2}(x)$ and $K_{\lambda_1}(x,y)\geq K_{\lambda_2}(x,y)$.
\item[(ii)]For all $\lambda_0> -\underline{D}$, $\lim_{\lambda\rightarrow\lambda_0}\Vert r_{\lambda} -r_{\lambda_0}\Vert_{\infty}=0$ and $\lim_{\lambda\rightarrow\lambda_0}\Vert K_{\lambda} - K_{\lambda_0}\Vert_{\infty}=0$
\item[(iii)]For all $\lambda>-\underline{D}$, let us denote $\overline{r}_{\lambda}:=\sup_{x\in\cS}r_{\lambda}(x)$. Then the map $\lambda\in\left]-\underline{D},+\infty\right[\longmapsto \overline{r}_{\lambda}$ is continuous and (strictly) decreasing.
\end{itemize}
\end{itemize}
\end{lemm} 
\begin{proof}[Proof of Lemma \ref{lemm:tech1}]
1): Let $\lambda> -\underline{D}$, we have
\begin{equation*}
\int_{\R_+}\exp\left(-\int_{0}^{a}D(x,\alpha)d\alpha -\lambda a\right)da\leq \int_{\R_+}\exp\left(-\underline{D}a -\lambda a\right)da <+\infty.
\end{equation*}
So, $r_\lambda$ and $K_\lambda$ are well-defined. Moreover, $r_{\lambda}$ and $K_\lambda$ are continuous by dominated convergence theorem and respectively positive and non-negative by Assumptions \ref{hypo1}. The second part of the assertion is a consequence of assumption \ref{hypo1}.
\\2) (i): Let $\lambda>-\underline{D}$ and $x\in\cS$. By derivation under the integral, 
\begin{equation*}
\partial_\lambda r_\lambda(x)=-(1-p)\int_{\R_+} aB(x,a)\exp\left(-\int_{0}^{a}D(x,\alpha)d\alpha-\lambda a\right)da <0.
\end{equation*}
The proof is similar for $K_{\lambda}$. 
\\(ii): Let $\lambda_0>-\underline{D}$ and $\alpha\in \left]-\underline{D},\lambda_0\right[$. For all $x\in \cS$ and $\lambda > \alpha$, we have
\begin{align*}
\vert r_{\lambda}(x)-r_{\lambda_0}(x)\vert &\leq \int_{0}^{+\infty}B(x,a)\vert e^{-(\lambda+\underline{D})a}-e^{-(\lambda_0 +\underline{D})a}\vert da \\
&\leq \vert \lambda -\lambda_0\vert \Vert B\Vert_{\infty}\int_{0}^{+\infty}a e^{-(\alpha +\underline{D}) a}da
\end{align*}
where we used that $\vert e^{-x} -e^{-y}\vert \leq \vert x -y\vert$ if $x,y\geq 0$, and which allow us to conclude. The proof is similar for $K_\lambda$.
\\(iii): Let $\lambda_0>-\underline{D}$ and let $(\lambda_j)$ be a sequence such that $\lambda_j\rightarrow\lambda_0$. Let $(x_j)\in\cS$ such that $\overline{r}_{\lambda_j}=r_{\lambda_j}(x_j)$ and let denote $x^{*}\in\cS$ a limit point of $(x_j)$. Using (ii), we obtain that $\overline{r}_{\lambda_j}\rightarrow r_{\lambda_0}(x^{*})=\overline{r}_{\lambda_0}$. We conclude by proving the strict monotonicity. Let $-\underline{D}< \lambda_1<\lambda_2$. For all $x\in\cS$ we have $r_{\lambda_2}(x)<r_{\lambda_1}(x)$, and so $r_{\lambda_2}(x)<\overline{r}_{\lambda_1}$. Since $\cS$ is compact and $r_{\lambda_2}$ is continuous, we get that $\overline{r}_{\lambda_2}<\overline{r}_{\lambda_1}$.
\end{proof}
For any $\lambda\in\left]-\underline{D},+\infty\right[$ we define the operators $\tilde{r}_\lambda + \tilde{J}_{\lambda},\tilde{r}_\lambda + \tilde{G}_{\lambda}:\mathcal{M}(\cS)\rightarrow\mathcal{M}(\cS)$ by  
\begin{equation}\label{eq:defopmeasure2}
(\tilde{r}_\lambda+\tilde{J}_\lambda)\mu =r_\lambda(x)\mu + \left(\int_{\cS}K_\lambda(y,x)\mu(dy)\right)dx,
\end{equation}
\begin{equation}\label{eq:defopmeasure3}
(\tilde{r}_\lambda+\tilde{G}_\lambda)\mu =r_\lambda(x)\mu + \left(\int_{\cS}K_\lambda(x,y)\mu(dy)\right)dx
\end{equation}
and $r_\lambda +J_\lambda, r_\lambda +G_\lambda:C(\cS)\rightarrow C(\cS)$ similarly as (\ref{eq:opintrocoville}) and (\ref{eq:opdualintrocoville}).
By Lemma \ref{lemm:tech1}, we deduce that Theorem \ref{prop:eigenintegro} is satisfied for $r_\lambda$ and $K_\lambda$. We recall the conclusions in the following lemma.
\begin{lemm}\label{prop:eigenintegroparam}
Assume Assumptions \ref{hypo1}. For all $\lambda >-\underline{D}$, there exists $\mu_\lambda\in \mathcal{M}^{+}(\cS)$ such that
\begin{equation*}
\rho(r_\lambda+G_\lambda)\mu_\lambda =(\tilde{r}_\lambda+\tilde{J}_\lambda)\mu_\lambda
\end{equation*}
which satisfies $\mu_\lambda(A)>0$ for all $A$ Borel subset of $\cS$ such that $\text{Leb}(A)>0$. Moreover, let us denote $\overline{r}_\lambda =\sup_{x\in\cS}r_\lambda(x)$ and $\Sigma_\lambda=\lbrace x\in\cS:r_\lambda(x)=\overline{r}_\lambda\rbrace$. Then we have:
\begin{itemize}
\item[(i)]$\overline{r}_\lambda\leq \rho(r_\lambda+G_\lambda)$.
\item[(ii)]$\overline{r}_\lambda<\rho(r_\lambda+G_\lambda)$ if and only if there exists $u_\lambda\in C(\cS)$, $u_\lambda >0$ such that $\mu_\lambda =u_\lambda(x)dx$. In this case, $\rho(r_\lambda+G_\lambda)$ is an eigenvalue of $r_\lambda+J_\lambda$ with algebraic multiplicity equals to one, associated with the eigenfunction $u$.
\item[(iii)]If $\text{Leb}(\Sigma_\lambda)> 0$ or if, $\text{Leb}(\Sigma_\lambda)=0$ and $\frac{1}{\overline{r}_\lambda-r_\lambda}\notin L^1(\cS)$, we have $\overline{r}_\lambda<\rho(r_\lambda+G_\lambda)$.
\item[(iv)]If $\rho(r_\lambda+G_\lambda)=\overline{r}_\lambda$, we have $\mu_\lambda=\mu^s +h(x)dx$ with $h\in L^1(\cS)$ and either $\mu^s=0$, or $\mu^s\neq 0$ and $\text{supp}(\mu^s)\subset \Sigma$.
\item[(v)]$\rho(r_\lambda+J_\lambda)=\rho(r_\lambda+G_\lambda)=\rho(\tilde{r}_\lambda+\tilde{J}_\lambda)=\rho(\tilde{r}_\lambda+\tilde{G}_\lambda)$.
\end{itemize}
Moreover, the same results are true exchanging $J_\lambda$ and $G_\lambda$, $\tilde{J}_\lambda$ and $\tilde{G}_\lambda$.
\end{lemm}
\begin{proof}
The proof of this theorem is a direct  consequence of Theorem \ref{prop:eigenintegro}. 
\end{proof}
 The following assumption allows us to characterise the Malthusian parameter associated with the linear dynamics.
\begin{hypo}\label{hypo3}
There is $\underline{\lambda}>-\underline{D}$ such that $\rho(\tilde{r}_{\underline{\lambda}}+\tilde{J}_{\underline{\lambda}})>1$.
\end{hypo}
\begin{prop}\label{prop:malthus}
Assume Assumptions \ref{hypo1} and \ref{hypo3}. The map $\lambda\in\left[\underline{\lambda},+\infty\right[\longmapsto \rho(\tilde{r}_\lambda +\tilde{J}_\lambda)$ is continuous and (strictly) decreasing. There exists a unique $\lambda^{*}\in \left[\underline{\lambda}, +\infty\right[$ such that $\rho(\tilde{r}_{\lambda^{*}}+\tilde{J}_{\lambda^{*}})=1$.
\end{prop}
\begin{proof}[Proof of Proposition \ref{prop:malthus}]
First, remark that by Lemma \ref{prop:eigenintegroparam} (v), we have for any $\lambda\in\left[\underline{\lambda},+\infty\right[$, $\rho(\tilde{r}_\lambda +\tilde{J}_\lambda)=\rho(r_\lambda +J_\lambda)$.
We divide the proof in three steps.
\\
\textbf{Step 1: the map $\lambda\in\left[\underline{\lambda}, +\infty\right[\longmapsto \rho(r_\lambda +J_{\lambda})$ is non-increasing.}
Let $\underline{\lambda}\leq\lambda_1\leq \lambda_2$. By Lemma \ref{lemm:tech1} 2)(i), we get that for all $f\in C^{+}(\cS)$, 
\begin{equation}
(r_{\lambda_2}+J_{\lambda_2})f(x)\leq (r_{\lambda_1}+J_{\lambda_1})f(x).
\end{equation}
By Proposition \ref{annex:montoto} (i) in Appendix, the spectral radius is monotone on the set of positive operators, we conclude that $\rho(r_{\lambda_2}+J_{\lambda_2})\leq \rho(r_{\lambda_1}+J_{\lambda_1})$.
\\
\textbf{Step 2: The map $\lambda\in\left[\underline{\lambda}, +\infty\right[\longmapsto \rho(r_\lambda +J_{\lambda})$ is continuous.} Let $\lambda_0\geq \underline{\lambda}$. We consider the two possible cases. First, we assume that $\overline{r}_{\lambda_0}<\rho(r_{\lambda_0}+J_{\lambda_0})$. We deduce by Lemma \ref{prop:eigenintegroparam} (ii) that :  $\rho(r_{\lambda_0}+J_{\lambda_0})$ is an eigenvalue of $r_{\lambda_0} +J_{\lambda_0}$ with algebraic multiplicity equals to one. On the other hand, for all $\lambda\geq \underline{\lambda}$ and $f\in C(\cS)$, we have
\begin{equation}\label{proo:inegnorm}
\Vert (r_{\lambda}+J_{\lambda})f - (r_{\lambda_0}+J_{\lambda_0})f\Vert_\infty\leq (\Vert r_{\lambda}-r_{\lambda_0}\Vert_{\infty}+\text{Leb}(\cS)\Vert K_{\lambda}-K_{\lambda_0}\Vert_{\infty})\Vert f\Vert_{\infty}.
\end{equation}
Therefore, we deduce from Proposition \ref{annex:kato} that: \textbf{(a)} there is $\delta>0$ such that if $\vert\lambda -\lambda_0\vert <\delta$, there exists an eigenvalue $\kappa_\lambda$  of $r_\lambda +J_\lambda$ with algebraic multiplicity equals to one; \textbf{(b)}  $P_\lambda \longrightarrow P$ as $\lambda\rightarrow \lambda_0$ for the operator norm where $P$, $P_\lambda$ represent respectively the projector  on the null space of $\rho(r_{\lambda_0}+J_{\lambda_0})I - r_{\lambda_0}+J_{\lambda_0}$ and $\kappa_\lambda I - r_{\lambda}+J_{\lambda}$. Let $u$ be a positive eigenfunction of $r_{\lambda_0} +J_{\lambda_0}$ associated with $\rho(r_{\lambda_0}+J_{\lambda_0})$. By \textbf{(b)}, we have $P_\lambda u\rightarrow Pu = u$ when $\lambda\rightarrow\lambda_0$. In particular, we deduce that there is $0<\delta'<\delta$ such that if $\vert \lambda -\lambda_0\vert <\delta'$, $P_\lambda u$ is a positive eigenfunction of $r_\lambda + J_\lambda$ associated with $\kappa_\lambda$. Hence, Proposition \ref{annex:spectral} (i) gives that if $\vert\lambda -\lambda_0\vert<\delta'$, $\rho(r_\lambda+J_\lambda)=\kappa_\lambda$. In order to conclude, let $(\lambda_j)$ be a sequence of $\left[\underline{\lambda},+\infty\right[$ which converges to $\lambda_0$. Since the function $\lambda\in \left[\underline{\lambda},+\infty\right[\rightarrow\rho(r_\lambda +J_\lambda)$ is bounded, there exists $\rho^{*}\in \left[0, +\infty\right[$ and a subsequence always denoted $(\lambda_j)$ such that $\rho(r_{\lambda_j}+J_{\lambda_j})\rightarrow \rho^{*}$. We check that $\rho^{*}=\rho(r_{\lambda_0
}+J_{\lambda_0})$. Let $u$ be a positive eigenfunction of $r_{\lambda_0}+J_{\lambda_0}$. For all $j$ sufficiently large, we have $(r_{\lambda_j}+J_{\lambda_j})P_{\lambda_j} u=\rho(r_{\lambda_j} +J_{\lambda_j}) P_{\lambda_j} u$.  By (\ref{proo:inegnorm})  and \textbf{(b)}, and taking the limit $j\rightarrow\infty$, we deduce that $(r_{\lambda_0}+J_{\lambda_0})u =\rho^{*} u$. So, $u$ is a positive eigenfunction associated with the eigenvalue $\rho^{*}$, it comes from Proposition \ref{annex:spectral} (i) that $\rho^{*}=\rho(r_{\lambda_0}+J_{\lambda_0})$. Consider now the case where $\overline{r}_{\lambda_0}=\rho(r_{\lambda_0}+J_{\lambda_0})$. Assume that the map $\lambda\in\left[\underline{\lambda},+\infty\right[\longmapsto\rho(r_{\lambda}+J_{\lambda})$ is not continuous at $\lambda_0$. So we have $\liminf_{\lambda\rightarrow\lambda_0}\rho(r_{\lambda}+J_{\lambda})<\limsup_{\lambda\rightarrow\lambda_0}\rho(r_{\lambda_0}+J_{\lambda_0})$.  Since the spectral radius is upper semi-continuous (see Lemma \ref{lemm:uppersemi} in Appendix A), we deduce that 
\begin{equation*}
\overline{r}_{\lambda_0}=\liminf_{\lambda\rightarrow\lambda_0}\rho_e(r_{\lambda}+J_{\lambda})\leq\liminf_{\lambda\rightarrow\lambda_0}\rho(r_{\lambda}+J_{\lambda})<\rho(r_{\lambda_0}+J_{\lambda_0})=\overline{r}_{\lambda_0}
\end{equation*}
which is absurd and that concludes the proof of the continuity.
\\\textbf{Step 3: strong monotonicity.} Let $\underline{\lambda}\leq \lambda_1\leq\lambda_2$. By the first part of this proof, we have $\rho(r_{\lambda_1}+J_{\lambda_1})\geq \rho(r_{\lambda_2}+J_{\lambda_2})$. Assume that $\rho(r_{\lambda_1}+J_{\lambda_1})= \rho(r_{\lambda_2}+J_{\lambda_2})$. We show that necessarily $\lambda_1=\lambda_2$. As before, we distinguish two cases. First we consider the case where $\rho(r_{\lambda_1}+J_{\lambda_1})>\overline{r}_{\lambda_1}$. Since $r_{\lambda_1}+J_{\lambda_1}$ is irreducible and $\rho(r_{\lambda_1}+J_{\lambda_1})$ is a pole of the resolvent of $r_{\lambda_1}+J_{\lambda_1}$, we deduce by Proposition \ref{annex:montoto} (ii) that $r_{\lambda_1}+J_{\lambda_1}=r_{\lambda_2}+J_{\lambda_2}$ and so $\lambda_1=\lambda_2$. We now consider the second case $\rho(r_{\lambda_1}+J_{\lambda_1})=\overline{r}_{\lambda_1}$. Assume that $\lambda_1<\lambda_2$. If there exists $\tilde{\lambda}\in\left[\lambda_1,\lambda_2\right]$ such that $\rho(r_{\tilde{\lambda}} +J_{\tilde{\lambda}})>\overline{r}_{\tilde{\lambda}}$ we conclude that $\rho(r_{\lambda_1} +J_{\lambda_1})\geq \rho(r_{\tilde{\lambda}}+J_{\tilde{\lambda}})>\rho(r_{\lambda_2}+J_{\lambda_2})$ by the previous part of the proof. Otherwise, we conclude that $\rho(r_{\lambda_{1}}+J_{\lambda_{1}})>\rho(r_{\lambda_2}+J_{\lambda_2})$ arguing that $\lambda\in\left[\underline{\lambda},+\infty\right[\mapsto \overline{r}_{\lambda}$ is decreasing (see Lemma \ref{lemm:tech1} 2)(iii)). We deduce that $\lambda_1=\lambda_2$. 
\\ Since $\rho(r_{\underline{\lambda}}+J_{\underline{\lambda}})>\rho_e(r_{\underline{\lambda}}+J_{\underline{\lambda}})=\overline{r}_{\underline{\lambda}}$, Assumption \ref{hypo3} gives that $\rho(r_{\underline{\lambda}}+J_{\underline{\lambda}})>1$ and we conclude by the intermediate value theorem. 
\end{proof}
The two following theorems give the existence of principal real eigenelements associated with the linear dynamics.
\begin{prop}\label{theo:main}
Assume Assumptions \ref{hypo1} and \ref{hypo3}. 
\begin{itemize}
\item[(i)] Let $\lambda \in\left[\underline{\lambda},\infty\right[.$ A non-zero measure $N\in \mathcal{M}^{+}(\cS\times\R_+)$ is a solution of
\begin{equation}\label{eq:eigenmeasure}
\forall f\in C_b^{0,1},\quad \int_{\cS\times\R_+}\left(\partial_a f-(D+\lambda)f+\mathcal{G}\left[f\right]\right)(x,a)N(dx,da)=0
\end{equation}
if and only if
\begin{equation*}
N(dx,da)=\mu_{\lambda}(dx) R_{\lambda}(x,a)da,
\end{equation*}
where $\mu_{\lambda}\in\mathcal{M}^{+}(\cS)$ is non-zero and satisfies $(\tilde{r}_{\lambda}+\tilde{J}_{\lambda})\mu_{\lambda}=\mu_{\lambda}$. 
\item[(ii)] The largest $\lambda\in\left[\underline{\lambda},+\infty\right[$ such that there exists a non-zero measure $N\in \mathcal {M}^{+}(\cS\times\R_+)$ which satisfies (\ref{eq:eigenmeasure}) is the unique solution of $\rho(\tilde{r}_{\lambda}+\tilde{J}_{\lambda})=1$.
\end{itemize}
\end{prop}
\begin{proof}[Proof of Proposition \ref{theo:main}]
(i): Let $\lambda\in\left[\underline{\lambda},+\infty\right[$ and let $N\in \mathcal{M}^{+}(\cS\times\R_+)$ be a non-trivial solution of (\ref{eq:eigenmeasure}). We decompose $N$ as
\begin{equation*}
N(dx,da)=\nu(dx)u(x,da).
\end{equation*}
where $\nu\in\mathcal{M}^{+}(\cS)$  and $u(x,da)$ be an associated transition measure. We extend these quantities to the whole set of the following way. Let $\tilde{\nu}\in \mathcal{M}^{+}(\R^d)$ defined for all Borel subset A of $\R^d$ by $\tilde{\nu}(A)=\nu(A\cap \cS)$; let $\tilde{u}(x,da)$ defined by $\tilde{u}(x,da)=0$ if $x\notin \cS$ and $
\tilde{u}(x,A)=u(x,A\cap \R_+)$ if $x\in\cS$ and $A$ is a Borel subset of $\R$. We define $\tilde{N}(dx,da)=\tilde{\nu}(dx)\tilde{u}(x,da)\in \mathcal{M}^{+}(\R^d\times\R)$. We extend continuously the functions $B,D$ and $k$ to the whole sets $\R^d\times\R$, $\R^d\times\R$ and $\R^d\times\R^d\times\R$ respectively and we denote by $\tilde{B}$, $\tilde{D}$ and $\tilde{k}$ their extensions. We define 
\[\tilde{V}=\frac{\mathds{1}_{\cS\times\R_+}(x,a)}{\tilde{R}_{\lambda}(x,a)}\tilde{N}\in \mathcal{M}^{+}_{\text{loc}}(\R^d\times\R)\]
where $\tilde{R}_{\lambda}(x,a)\in C(\R^d\times\R)$ is defined similarly as $R_\lambda(x,a)$. Now we compute the distributional partial derivatives $\partial_a \tilde{V}$. Let $f\in C_c^{\infty}(\R^d\times\R)$,
\begin{align*}
\int_{\R^d\times\R}\partial_a f(x,a) \tilde{V}(dx,da) &=\int_{\cS\times\R_+}\left(\partial_a \frac{f}{R_{\lambda}}-f\frac{D+\lambda}{R_{\lambda}}\right)(x,a)N(dx,da)\\
&=-\int_{\cS\times\R_+}\cG\left[\frac{f}{R_\lambda}\right](x,a)N(dx,da)\\
&=-\int_{\R^d}f(x,0)(\tilde{P}(x)\tilde{\nu}(dx) +\mathds{1}_{\cS}(x)\tilde{Q}(x)dx)
\end{align*}
where we introduced
\begin{equation*}
\begin{cases}
\tilde{P}(x)=(1-p)\int_{\R}\tilde{B}(x,a)\tilde{u}(x,da)\\
\tilde{Q}(x)=p\int_{\R^d\times\R}\tilde{B}(y,a)\tilde{k}(y,a,x)\tilde{u}(y,da)\tilde{\nu}(dy).
\end{cases}
\end{equation*}
It comes that $\partial_a \tilde{V}=(\tilde{P}\tilde{\nu}+\mathds{1}_{\cS}\tilde{Q}dx)\delta_0(da)$. Since the primitives of the zero distribution are constant functions, we deduce that there exists a distribution $T\in \mathcal{D}'(\R^d)$ such that 
\begin{equation*}
\tilde{V}(dx,da)=(\tilde{P}\tilde{\nu}+\mathds{1}_{\cS}\tilde{Q}dx)\mathds{1}_{\R_+}(a)da + T(x).
\end{equation*}
Since the support of $\tilde{V}$ is a subset of $\cS\times\R_+$, we get that $T=0$ and finally 
\begin{equation}\label{eq:proo3}
N(dx,da)=(P(x)\nu(dx)+Q(x)dx)R_{\lambda}(x,a)da
\end{equation}
where $P:=\tilde{P}_{\vert \cS}$ and $Q:=\tilde{Q}_{\vert \cS}$. By (\ref{eq:proo3}), we obtain:
\begin{align}\label{pr:3}
P(x)\nu(dx)=(P(x)\nu(dx)+Q(x)dx)r_{\lambda}(x),
\end{align}
and
\begin{align}\label{pr:4}
Q(x)dx= \tilde{J}_{\lambda}(P\nu +Qdy)(dx).
\end{align}
Denoting $\mu:=P(x)\nu(dx)+Q(x)dx$, we get by (\ref{pr:3}) and (\ref{pr:4}) that $\mu =(\tilde{r}_{\lambda}+\tilde{J}_{\lambda})\mu$. Finally, it comes  by (\ref{eq:proo3}) that 
\begin{equation*}
N(dx,da)=\mu(dx)R_\lambda(x,a)da.
\end{equation*}
Reciprocally, it is easy to check that such a measure is solution of (\ref{eq:eigenmeasure}). 
\\(ii): Let $\lambda\in\left[\underline{\lambda},+\infty\right[$ and assume that there exists a non-zero $ N\in \mathcal{M}^{+}(\cS\times\R_+)$ which satisfies (\ref{eq:eigenmeasure}). By (i), we deduce that there is $\mu\in\mathcal{M}^{+}(\cS)$ such that $N=\mu(dx)R_{\lambda}(x,a)$ and which satisfies $(\tilde{r}_{\lambda}+\tilde{J}_{\lambda})\mu = \mu$. We deduce that $1\leq \rho(\tilde{r}_{\lambda}+\tilde{J}_{\lambda})$. Moreover, if $\rho(\tilde{r}_\lambda + \tilde{J}_\lambda)=1$ there exists $\nu \in\mathcal{M}^{+}(\cS)$ such that $(\tilde{r}_{\lambda}+\tilde{J}_{\lambda})\nu = \nu$ by Lemma \ref{prop:eigenintegroparam}. That concludes the proof.
\end{proof}
\begin{prop}\label{theo:maindual}
Assume Assumptions \ref{hypo1} and \ref{hypo3}.
\begin{itemize}
\item[(i)]Let $\lambda\in\left[\underline{\lambda},+\infty\right[$. A non-zero measure $\psi\in\mathcal{M}_{\text{loc}}^{+}(\cS\times\R_+)$ such that 
\begin{equation*}
\psi(dx,da)=\varphi(dx)m(x,a)da, \quad m(x,.)\in L^{\infty}(\R_+), \varphi(dx)\in\mathcal{M}^{+}(\cS)
\end{equation*}
is a solution of
\begin{align}\label{eigenmeasuredual}
\forall f \in C_c^{0,1},&\quad \int_{\cS\times\R_+}(\partial_a f+(D+\lambda)f)(x,a)\psi(dx,da)\\
&=\int_{\cS\times\R_+}f(x,a)\tilde{\cG}\left[\psi\right](dx,da)-\int_{\cS}m(x,0)f(x,0)\varphi(dx),\nonumber
\end{align}
where
\begin{equation*}
\tilde{\cG}\left[\psi\right]=B(x,a)\left((1-p)\varphi(dx)m(x,0)+p\int_{\cS}\varphi(dy)m(y,0)k(x,a,y)\right)da,
\end{equation*}
if and only if
\begin{align*}
\psi(dx,da)&=\frac{da}{R_{\lambda}(x,a)}\left((1-p)\eta_{\lambda}(dx)\int_{a}^{+\infty}B(x,\alpha)R_{\lambda}(x,\alpha)d\alpha\right. \\
&\left.+p\int_{\cS}\eta_{\lambda}(dy)\int_{a}^{+\infty}B(x,\alpha)R_{\lambda}(x,\alpha)k(x,\alpha,y)d\alpha\right),
\end{align*}
where $\eta_{\lambda}\in\mathcal{M}^{+}(\cS)$ is non-zero and satisfies $(\tilde{r}_{\lambda}+\tilde{G}_{\lambda})\eta_{\lambda}=\eta_{\lambda}$. 
\item[(ii)] The largest $\lambda\in\left[\underline{\lambda},+\infty\right[$ such that there exists a non-zero $ \psi\in \mathcal{M}^{+}_{\text{loc}}(\cS\times\R_+)$ which satisfies (\ref{eigenmeasuredual}) is the unique solution of $\rho(\tilde{r}_\lambda +\tilde{J}_\lambda)=1$.
\end{itemize}
\end{prop}
\begin{proof}[Proof of Proposition \ref{theo:maindual}]
The proof is very similar to the previous proof.
\\(i): Let $\lambda\in \left[\underline{\lambda},+\infty\right[$ and let $\psi\in\mathcal{M}_{\text{loc}}^{+}(\cS\times\R_+)$   such that
\begin{equation*}
\psi=\varphi(dx)m(x,a)da
\end{equation*}
with $m(x,.)\in L^{\infty}(\R_+)$ and $\varphi\in \mathcal{M}^{+}(\cS)$, be a non-trivial solution of (\ref{eigenmeasuredual}). As in the previous proof, we extend all the quantities to the whole set $\R^d\times\R$. We denote 
\[\tilde{U}=\mathds{1}_{\cS\times\R_+}(x,a)\tilde{R}_{\lambda}(x,a)\tilde{\psi}\in \mathcal{M}^{+}_{\text{loc}}(\R^d\times\R)\]
and we compute the partial distributional derivative $\partial_a \tilde{U}$. We denote $\eta(dx)=\varphi(dx)m(x,0)$. Let $g\in C_c^{\infty}(\R^d\times\R)$, we have
\begin{align*}
\int_{\R^d\times\R}\partial_a g(x,a)\tilde{U}(dx,da)&=\int_{\cS\times\R_+}\left(\partial_a (gR_\lambda) +(D+\lambda)(gR_\lambda)\right)(x,a)\psi(dx,da)\\
&=-\int_{\cS\times\R_+}g(x,a)\tilde{R}_\lambda(x,a)\tilde{\cG}\left[\psi\right](dx,da)+\int_{\cS}g(x,0)\eta(dx).
\end{align*}
We deduce that $\partial_a \tilde{U}= \left(\eta \delta_0-\tilde{R}_{\lambda}\tilde{\cG}\left[\psi\right]\right)\mathds{1}_{\cS\times\R_+}$
and
\begin{equation*}
\tilde{U}=\left(\eta(dx)-\int_{0}^{a}R_\lambda(x,\alpha)\tilde{\cG}\left[\psi\right](dx,d\alpha)\right)\mathds{1}_{\cS\times\R_+}(x,a)da.
\end{equation*}
It comes that 
\begin{equation}\label{pr:0}
\psi=\frac{1}{R_{\lambda}(x,a)}\left(\eta(dx)-\int_{0}^{a}R_\lambda(x,\alpha)\tilde{\cG}\left[\psi\right](dx,da)\right).
\end{equation}
Since $\psi\in\mathcal{M}^{+}_{\text{loc}}(\cS\times\R_+)$ with $\psi=\varphi(dx)m(x,a)da$ and $m(x,.)\in L^{\infty}(\R_+)$,
we have necessarily
\begin{equation}\label{pr:1}
\int_{0}^{+\infty}R_\lambda(x,\alpha)\tilde{\cG}\left[\psi\right](dx,d\alpha)=\eta(dx)
\end{equation} 
which is equivalent to 
\begin{equation}\label{pr:2}
r_{\lambda}(x)\eta(dx) +dx\int_{\cS}\eta(dy)K_{\lambda}(x,y) =\eta(dx).
\end{equation}
By (\ref{pr:2}) we get that $(\tilde{r}_\lambda +\tilde{G}_\lambda)\eta =\eta$ and using (\ref{pr:0}), (\ref{pr:1}), (\ref{pr:2}), it comes that 
\begin{align*}
\psi &=\frac{da}{R_{\lambda}(x,a)}\left((1-p)\eta(dx)\int_{a}^{+\infty}B(x,\alpha)R_{\lambda}(x,\alpha)d\alpha\right. \\
&\left.+p\int_{\cS}\eta(dy)\int_{a}^{+\infty}B(x,\alpha)R_{\lambda}(x,\alpha)k(x,\alpha,y)d\alpha\right).
\end{align*}
Conversely, it is easy to check that such a measure is a solution of (\ref{eigenmeasuredual}).
The proof of (ii) is similar as in the previous proof, since by Lemma \ref{prop:eigenintegroparam}, $\rho(\tilde{r}_\lambda +\tilde{J}_\lambda)=\rho(\tilde{r}_\lambda + \tilde{G}_\lambda)$.
\end{proof}
We deduce a corollary which concerns the "regular" case.
\begin{coro}\label{coroeigen}
Assume Assumptions \ref{hypo1}, \ref{hypo3} and that $\overline{r}_{\lambda^{*}}<1$. There exists a unique $(\lambda^{*},N,\phi)\in \left[\underline{\lambda}, +\infty\right[\times C(\cS,L^1(\R_+))\times C(\cS,L^{\infty}(\R_+))$ such that
\begin{equation}\label{eq:regeigendir}
\begin{cases}
-\partial_a N(x,a) - (D(x,a)+\lambda^{*})N(x,a) = 0, \quad (x,a)\in\cS\times\R_+,\\
N(x,0)=\cF\left[N\right](x),\quad \int_{\cS\times\R_+}N=1,
\end{cases}
\end{equation}
\begin{equation}\label{eq:regeigendual}
\begin{cases}
\partial_a \phi(x,a)-(D(x,a)+\lambda^{*})\phi(x,a) + \cG\left[\phi\right](x,a)=0,\quad (x,a)\in\cS\times\R_+,\\
\int_{\cS\times\R_+}N\phi=1.
\end{cases}
\end{equation}
\end{coro}
\begin{proof}[Proof of Corollary \ref{coroeigen}]
Let $\lambda^{*}\geq \underline{\lambda}$ be the unique solution of $\rho(r_{\lambda^{*}}+J_{\lambda^{*}})=1$ (see Proposition \ref{prop:malthus}). Since $\overline{r}_{\lambda^{*}}<1$ we deduce by Lemma \ref{prop:eigenintegroparam} (ii) that $1$ is a simple eigenvalue of the operators $r_{\lambda^{*}}+J_{\lambda^{*}}$ and $r_{\lambda^{*}}+G_{\lambda^{*}}$. Let $(N,\phi)\in C(\cS, L^1(\R_+))\times C(\cS, L^{\infty}(\R_+))$ be a solution of (\ref{eq:regeigendir}) and (\ref{eq:regeigendual}) with $\lambda=\lambda^{*}$. We  have $N(x,0), \phi(x,0)\in C^{+}(\cS)$ and we deduce by Theorem \ref{prop:eigenintegro} (ii) that $N(x,0)$ and $\phi(x,0)$ are positive eigenfunctions of $r_{\lambda^{*}}+J_{\lambda^{*}}$ and $r_{\lambda^{*}}+G_{\lambda^{*}}$  associated with the eigenvalue one. Since this eigenvalue is simple, the conditions $\int_{\cS\times\R_+}N=1$ and $\int_{\cS\times\R_+}N\phi=1$ allow us to fix $N$ and $\phi$. Assume now that there exists $(\lambda,N',\phi')$ such that $\lambda\neq\lambda^{*}$ and which satisfies (\ref{eq:regeigendir}) and (\ref{eq:regeigendual}). By Theorems \ref{theo:main} and \ref{theo:maindual}, it comes that $\lambda < \lambda^{*}$. We deduce that $\rho(r_\lambda +J_\lambda) > 1$. Moreover $N'(x,0),\phi'(x,0)\in C^{+}(\cS)$ are eigenfunctions of $r_\lambda + J_\lambda$ and $r_\lambda +G_\lambda$ associated with the eigenvalue 1, which is absurd by Proposition \ref{prop:schaefermain} (ii). 
\end{proof}
We are now able to describe the long-time behaviour of the solutions of the linear equation (\ref{eq:edplin}).
\begin{prop}\label{prop:asympdet}
Assume Assumptions \ref{hypo1}, \ref{hypo2}, \ref{hypo3} and that $\overline{r}_{\lambda^{*}}<1$. Assume that there exists $\underline{\eta}>0$ such that for all $a\in\R_+$ and $(x,y)\in\cS^2$
\begin{equation}\label{eq:hypocontr}
pB(y,a)k(y,x,a)\phi(x,0)\geq \underline{\eta}\phi(y,a).
\end{equation}
Let $(v_t)_{t\geq 0}\in C(\R_+, \mathcal{M}^{+}(\cS\times\R_+))$ be the solution of (\ref{eq:edplin}) started at $v_0\in \mathcal{M}^{+}(\cS\times\R_+)$. Then we have
\begin{align}\label{eq:inegexp}
\int_{\cS\times\R_+} \phi(x,a) \vert e^{-\lambda^{*}t} v_t& -m_0 N\vert(dx,da) \nonumber\\
&\leq e^{-\underline{\eta}\text{Leb}(\cS)t}\int_{\cS\times\R_+}\phi(x,a)\vert v_0-m_0N\vert(dx,da) 
\end{align}
where $m_0=\int_{\cS\times\R_+}\phi(x,a)v_0(dx,da)$.
\end{prop}
The idea of the proof is similar as in \cite{gwiazda2016generalized}. We show the spectral gap property for initial regular data, and then deduce it for measure initial data, using the two following Lemmas. The assumption (\ref{eq:hypocontr}) is similar as in \cite[Theorem 3.5]{perthame2006transport}. For any measure $\mu\in\mathcal{M}(\cS\times\R_+)$, let $\mu=\mu^a(x)dx +\mu^s(dx)$ be its Lebesgue decomposition. We define 
\begin{equation*}
\langle \mu\rangle = \int_{\cS\times\R_+}\sqrt{1 + \vert \mu^a(x)\vert^2}dx +\mu^s(\cS\times\R_+).
\end{equation*}
We use the following (semi-)continuity properties. The point (i) is well-known. See   \cite[Theorem 5]{kristensen2010relaxation} for (ii).
\begin{lemm}\label{propapproxmeasure}
Let $\mu_n,\mu \in \mathcal{M}(\cS\times\R_+)$ such that $\mu_n\underset{n\rightarrow\infty}{\rightarrow} \mu$ $\text{weakly}^*$ in $(C_b(\cS\times\R_+))'$. Then for all $f\in C_b(\cS\times\R_+)$,
\begin{itemize}
\item[(i)]\begin{equation*}
\liminf_{n\rightarrow\infty}\int_{\cS\times\R_+}f(x,a)\vert \mu_n\vert(dx,da)\geq \int_{\cS\times\R_+}f(x,a)\vert \mu\vert(dx,da).
\end{equation*}
\item[(ii)]If moreover $\langle\mu_n\rangle\underset{n\rightarrow\infty}{\rightarrow}\langle \mu\rangle$, then
\begin{equation*}
\lim_{n\rightarrow\infty}\int_{\cS\times\R_+}f(x,a)\vert \mu_n\vert(dx,da)= \int_{\cS\times\R_+}f(x,a)\vert \mu\vert(dx,da).
\end{equation*}
\end{itemize}
\end{lemm}
 The following lemma is proved in \cite[Theorem 2.4]{canizo2013measure}.
\begin{lemm}\label{lemm:canizo}
Assume Assumptions \ref{hypo1}, \ref{hypo2}. There is a constant $C>0$ such that for all $v_0^1,v_0^2\in \mathcal{M}^{+}(\cS\times\R_+)$ and $t\geq 0$:
\begin{equation*}
\Vert v_t^1 - v_t^2\Vert_{\text{BL}} \leq e^{Ct}\Vert v^1_0 -v^2_0\Vert_{\text{BL}} 
\end{equation*}
where $v^1, v^2$ are the solutions of (\ref{eq:edplin}) started at $v_0^1,v_0^2$.
\end{lemm}
We now give the proof of Proposition \ref{prop:asympdet}.
\begin{proof}[Proof of Proposition \ref{prop:asympdet}]
It suffices to prove the result for regular initial data, and then conclude by regularising the initial measure. To analyse the regular case, we follow the ideas of \cite{perthame2006transport}. So, let $(v_t)_{t\geq 0}$ be the solution of (\ref{eq:edplin}) started at $v_0\in L^1(\cS\times\R_+)\cap C^{0,1}_b$. First, by choosing $\phi$ as test function in (\ref{eq:weakformpde}), we obtain the following invariance:
\begin{equation}\label{proo:invariant}
\int_{\cS\times\R_+}e^{-\lambda^{*}t}v_t(x,a)\phi(x,a)dxda =\int_{\cS\times\R_+}v_0(x,a)\phi(x,a)dxda
\end{equation}
where $\lambda^{*}, \phi$ are defined in Corollary \ref{coroeigen}.
We now define  $h_t(x,a):=e^{-\lambda^{*} t}v_t(x,a) -m_0 N(x,a)$ where 
\begin{equation*}
m_0=\int_{\cS\times\R_+}v_0(x,a)\phi(x,a)dxda.
\end{equation*}
It is straightforward to verify that 
\begin{equation*}
\begin{cases}
\partial_t (h_t(x,a)\phi(x,a)) +\partial_a (h_t(x,a)\phi(x,a))= -h_t(x,a)\mathcal{G}\left[\phi\right](x,a)\\
h_t(x,0)\phi(x,0)=\phi(x,0)\mathcal{F}\left[h_t\right](x).
\end{cases}
\end{equation*}
Using a regularisation method used in \cite[Proposition 6.3]{perthame2006transport}, we deduce that 
\begin{equation}\label{eq:proofasymp1}
\begin{cases}
\partial_t (\vert h_t\vert(x,a)\phi(x,a)) +\partial_a (\vert h_t\vert(x,a)\phi(x,a))= -\vert h_t\vert(x,a)\mathcal{G}\left[\phi\right](x,a)\\
\vert h_t(x,0)\vert\phi(x,0)=\phi(x,0)\vert\mathcal{F}\left[h_t\right](x)\vert.
\end{cases}
\end{equation}
Integrating the first equation in (\ref{eq:proofasymp1}) over $\cS\times\R_+$, we obtain
\begin{align}\label{eq:prooasymp2}
\frac{\text{d}}{\text{dt}}\int_{\cS\times\R_+}\vert h_t(x&,a)\vert  \phi(x,a)dxda \\\nonumber
&=\int_{\cS}\phi(x,0)\vert \mathcal{F}\left[h_t\right](x)\vert dx - \int_{\cS\times\R_+} \vert h_t(x,a)\vert \mathcal{G}\left[\phi\right](x,a)dxda.
\end{align} 
From the equation (\ref{eq:prooasymp2}), and using the invariant (\ref{proo:invariant}), we deduce that
\begin{align*}\frac{\text{d}}{\text{dt}}&\int_{\cS\times\R_+}\vert h_t(x,a)\vert  \phi(x,a)dxda \\
&=\int_{\cS}\vert (1-p)\phi(x,0)\int_{\R_+}B(x,a)h_t(x,a)da \\
&+\int_{\cS\times\R_+}(pB(y,a)k(y,x,a)\phi(x,0) -\underline{\eta}\phi(y,a))h_t(y,a)dady\vert dx\\
&-\int_{\cS}\left((1-p)\int_{\R_+}\phi(x,0)B(x,a)\vert h_t(x,a)\vert da \right.\\
&\left.+p \int_{\cS\times\R_+} B(x,a)k(x,y,a)\phi(y,0)\vert h_t(x,a)\vert dyda\right)dx
\end{align*}
Since we assume that for any $a\in\R_+$, $(x,y)\in\cS^2$
\[
pB(y,a)k(y,x,a)\phi(x,0)\geq \underline{\eta}\phi(y,a),
\]
it comes that
\begin{equation*}
\frac{\text{d}}{\text{dt}}\int_{\cS\times\R_+}\vert h_t(x,a)\vert  \phi(x,a)dxda\leq -\underline{\eta}\text{ }\text{Leb}(\cS)\int_{\cS\times\R_+}\vert h_t(x,a)\vert \phi(x,a)dxda 
\end{equation*}
and we conclude by Gronwall's Lemma. Now consider the case $v_0\in \mathcal{M}^{+}(\cS\times\R_+)$. There is a family of non-negative functions $v_0^\epsilon\in C^{0,1}_b\cap L^1(\cS\times\R_+)$ such that $\Vert v_0^{\epsilon} -v_0\Vert_{\text{BL}}\underset{\epsilon \rightarrow0 }{\rightarrow} 0$ and $\langle v_0^{\epsilon}\rangle \underset{\epsilon \rightarrow 0 }{\rightarrow} \langle v_0\rangle$. By Lemma \ref{lemm:canizo}, it comes that for each $t\in\R_+$, $\Vert v_t^{\epsilon} - v_t\Vert_{\text{BL}}\underset{\epsilon\rightarrow 0}{\rightarrow}0$. 
\\Similarly as in \cite{gwiazda2016generalized}, we conclude by using Lemma \ref{propapproxmeasure}. By the previous part of the proof we deduce that (\ref{eq:inegexp}) is satisfied fo $v_t^{\epsilon}$. By applying Lemma \ref{propapproxmeasure} (i) with $\mu^{\epsilon} =  e^{-\lambda^{*}t} v_t^{\epsilon} -m_0^{\epsilon}N$ which satisfies $\mu^{\epsilon} \rightarrow e^{-\lambda^{*}t}v_t - m^{0}N$ weakly* as $\epsilon\rightarrow 0$; and Lemma \ref{propapproxmeasure} (ii) with $\mu_0^{\epsilon}= v_0^{\epsilon}-m_0^{\epsilon}N$ which satisfies $\mu_0^{\epsilon}\rightarrow v_0- m^{0}N$ weakly* and $\langle \mu_0^{\epsilon}\rangle\rightarrow\langle v_0- m^{0}N\rangle$ as $\epsilon\rightarrow 0$,  we conclude by taking the $\liminf$ in the left side of (\ref{eq:inegexp}) and the limit $\epsilon\rightarrow 0$ in the right side of (\ref{eq:inegexp}).
\end{proof}
Let us prove that the Malthusian parameter also plays a main role for the stochastic underlying dynamics. For each $K\in\N^{*}$, let $(Y_t^K)_{t\geq 0}$ be the process with infinitesimal generator $B^K$ defined, for any $f\in C_b^{0,1}$ and $\mu\in\mathcal{M}^{+}(\cS\times\R_+)$ by
\begin{align}\label{eq:generateurlin}
B^K F_f&(\mu)=\int_{\cS\times\R_+}\partial_a f(x,a)F'(\langle \mu,f\rangle)\mu(dx,da)\\
&+K\int_{\cS\times\R_+}\left\lbrace(F(\langle \mu +\frac{\delta_{(x,0)}}{K},f\rangle)-F(\langle \mu,f\rangle))(1-p)B(x,a)\right.\nonumber\\
&\left.+\left(\int_{\cS}(F(\langle \mu +\frac{\delta_{(y,0)}}{K},f\rangle)-F(\langle \mu,f\rangle))B(x,a)pk(x,a,y)dy\right)\right.\nonumber\\
&\left. +(F(\langle \mu -\frac{\delta_{(x,a)}}{K},f\rangle)-F(\langle \mu,f\rangle))D(x,a)\right\rbrace\mu(dx,da)\nonumber
\end{align}
where $F_f(\mu):=F(\langle \mu,f\rangle)$. The following result is similar as one proved in \cite{jagers2000population} for age structured dynamics.
\begin{prop}
Assume that $\lambda^{*}>0$ and there exists $C>0$ such that for all $(x,a)\in \cS\times\R_+$
\begin{equation}\label{eq:hypotechsto}
\mathcal{G}\left[\phi^{2}\right](x,a)+D(x,a)\phi^2(x,a)\leq C \phi(x,a).
\end{equation}
\begin{itemize}
\item[(i)] Let $K\in\mathbb{N}^{*}$. The process $V^K$ defined by $V_t^K:=e^{-\lambda^{*}t}\langle Y_t^K,\phi\rangle$ is a square integrable martingale which satisfies $\mathbb{E}\left[\langle V^K\rangle_{\infty}\right]<+\infty$. The process $V^K$ converges in $L^2$ and almost surely to a non degenerate limit $V_\infty^K$.
\item[(ii)]Assume that $\sup_{K\in\N^{*}}\mathbb{E}\left[\langle Y_0^K,1\rangle^2\right]<+\infty$ and that $Y_0^K$ converges in law to $v_0\in\mathcal{M}^{+}(\cS\times\R_+)$ as $K\rightarrow\infty$. For all $\epsilon>0$,
\begin{equation*}
\lim_{K\rightarrow\infty}\mathbb{P}\left(\sup_{t\geq 0}\vert V_t^K -\langle v_0,\phi\rangle\vert >\epsilon\right)=0.
\end{equation*}
\end{itemize}
\end{prop}
\begin{proof}
(i): Using a classical semimartingale decomposition for the process $Y^K$ proved in \cite{tran2006modeles}, the process $V^K$ is a square integrable martingale with quadratic variation 
\begin{align*}
\langle V^{K}\rangle_t = \frac{1}{K}\int_{0}^{t}e^{-2\lambda^{*}s}\langle Y_s^{K},\mathcal{G}\left[\phi^2\right]+D\phi^2\rangle ds.
\end{align*}
Taking the expectation, we deduce from (\ref{eq:hypotechsto}) that
\begin{align*}
\mathbb{E}\left[\langle V^K\rangle_t\right]& \leq  \frac{C}{K}\int_{0}^{t}e^{-\lambda^{*}s}\mathbb{E}\left[e^{-\lambda^{*}s}\langle Y_s^K,\phi\rangle\right] ds.
\end{align*}
Since $e^{-\lambda^{*}t}\langle Y_t^K,\phi\rangle$ is a martingale, we deduce that
\begin{align*}
\mathbb{E}\left[\langle V^K\rangle_t\right]\leq \frac{C \mathbb{E}\left[\langle Y_0^K,\phi\rangle\right]}{K}\int_{0}^{t}e^{-\lambda^{*}s}ds.
\end{align*}
Therefore
\begin{equation}\label{eq:crochetinf}
\mathbb{E}\left[\langle V^K\rangle _{\infty}\right]=O\left(\frac{1}{K}\right),
\end{equation}
which proves (i). 
\\(ii): Let $\epsilon>0$. We have
\begin{align*}
\mathbb{P}\left(\sup_{t\geq 0}\vert \langle V_t^K,\phi\rangle - \langle v_0,\phi\rangle\vert >\epsilon\right)\leq \mathbb{P}\left(\sup_{t\geq 0}\vert M_t^K\vert >\epsilon   \right)+\mathbb{P}\left(\vert V_0^K -\langle v_0,\phi\rangle\vert >\epsilon\right)
\end{align*}
where $M_t^K$ is a martingale started at $0$ which satisfies $\mathbb{E}\left[\langle M^K\rangle _{\infty}\right]=O\left(\frac{1}{K}\right)$. By Doob's inequality, we deduce that
\begin{align*}
\mathbb{P}\left(\sup_{t\geq 0}\vert M_t^K\vert >\epsilon   \right)\leq \frac{2}{\epsilon^2}\sup_{t\geq 0}\mathbb{E}\left[\vert M_t^K\vert^2\right]= \frac{2}{\epsilon^2}\mathbb{E}\left[\langle M^K\rangle_{\infty}\right]
\end{align*}
and we conclude using (\ref{eq:crochetinf}).
\end{proof}
As in \cite{bonnefon2015concentration}, we can give a concrete example in which the eigenmeasure is singular. One can note that this phenomenon appears only if $p<1$. If $p=1$, we obtain that $r_\lambda = 0$: the operator $r_\lambda +J_\lambda = J_\lambda$ is compact and the eigenelements are continuous functions. 
\paragraph{An example of a non regular stable distribution.} Let $\cS\subset \R^d$ which satisfies (\ref{hyp:S}). Let $B(x,a)=B(x)$ such that $\frac{1}{\overline{B}-B}\in L^1(\cS)$. Let $D(x,a)=D\in\R_+^{*}$ and $k(x,y,a)=1$.
\begin{prop}
Let $p_0\in \left]0,1\right[$ such that $(p_0/(1-p_0))\overline{B}\int_{\cS}\frac{1}{\overline{B}-B(x)}dx<1$. 
Let $\lambda\geq \underline{\lambda}$, $\mu_\lambda =\mu^s + u(x)dx\in \mathcal{M}^{+}(\cS)$ and $\eta_\lambda =\eta^s + v(x)dx \in\mathcal{M}^{+}(\cS)$ be  such that $\mu_\lambda = (\tilde{r}_\lambda+\tilde{J}_\lambda)\mu_\lambda$ and $\eta_\lambda = (\tilde{r}_\lambda+\tilde{G}_\lambda)\eta_\lambda$.
Then for all $p \in \left]0,p_0\right[$, $\mu^s\neq 0$ and $\eta^s\neq 0$.
\end{prop}
\begin{proof}
Let $p\in \left]0,p_0\right[$ and $\lambda\geq \underline{\lambda}$. First remark that we have $r_\lambda(x)=(1-p)\frac{B(x)}{\lambda+D}$ and $K_{\lambda}(x,y)=p\frac{B(x)}{\lambda+D}$. Assume there exists a non negative function $v\in L^1(\cS)$ with $\int_{\cS}v(y)dy=1$ such that $(\tilde{r}_{\lambda}+\tilde{G}_{\lambda})v = v$. We have almost everywhere on $\cS$,
\begin{equation*}
r_{\lambda}(x)v(x)+ \int_{\cS}K_{\lambda}(x,y)v(y)dy =v(x).
\end{equation*}
We get that almost everywhere on $\cS$
\begin{equation*}
v(x)=p\frac{B(x)}{\lambda+D}\frac{1}{1-r_{\lambda}(x)}\leq p\frac{B(x)}{\lambda+D}\frac{1}{\overline{r}_{\lambda}-r_{\lambda}(x)}=\frac{p}{1-p}\frac{B(x)}{\overline{B}-B(x)}
\end{equation*}
and we deduce that 
\begin{equation*}
1\leq \frac{p}{1-p}\overline{B}\int_{\cS}\frac{1}{\overline{B}-B(x)}dx <1
\end{equation*}
which is absurd. So, $\eta^s\neq 0$. The proof is similar for $\mu^s$.
\end{proof}
\section{Proof of the Main Results}
We can now give the proof of our main results stated in Section 1.
\begin{proof}[Proof of Theorem \ref{maintheostatio}]
It is obvious that stationary states $\overline{n}\in \mathcal{M}^{+}(\cS\times\R_+)$ are eigenmeasures of the linear operator. Indeed, they are solutions of (\ref{eq:eigenmeasure}) with $\lambda = c\int_{\cS\times\R_+}\overline{n}$. Since $\rho(\tilde{r}_0 +\tilde{J}_0)>1$ we deduce by monotony that $\lambda^{*}>0$ and we can choose an eigenvector $\overline{n}=\mu_{\lambda^{*}}R_{\lambda^{*}}$ which satisfies $c\int_{\cS\times\R_+}\overline{n} =\lambda^{*}$. 
\end{proof}
\begin{proof}[Proof of Theorem \ref{maintheolongtime}]
The idea is to transform the solution of the non-linear equation to those of the linear equation. Indeed, let $(n_t)_{t\geq 0}\in C(\R_+, \mathcal{M}^{+}(\cS\times\R_+))$ be the solution of (\ref{eq:weakformpde}) and let us denote $\rho(t)=\int_{\cS\times\R_+}n_t(dx,da)$. It is straightforward to check that 
\begin{equation*}
v_t(dx,da)=\exp\left(c\int_{0}^{t}\rho(s)ds\right)n_t(dx,da)
\end{equation*}
is a weak solution of the linear equation
\begin{equation*}
\begin{cases}
\partial_t v_t(x,a)+\partial_a v_t(x,a)  =-D(x,a)v_t(x,a)\\
v_t(x,0) = \cF\left[v_t\right](x),\quad v_0=n_0.
\end{cases}
\end{equation*}
Let $(\lambda^{*},N,\phi)$ be the eigenelements given by Corollary \ref{coroeigen}. By the assumptions of the theorem we obtain  that 
\begin{align*}
\phi(x,a)\geq \frac{(1-p)\inf_{x\in \cS}\phi(x,0)\underline{B}}{\lambda^{*}+\Vert D\Vert_\infty}=:\underline{\phi} >0
\end{align*}
and that 
\begin{align*}
\frac{pB(x,a)k(x,a,y)\phi(x,0)}{\phi(x,a)}\geq \frac{p \underline{B}\underline{k}\underline{\phi}}{\Vert \phi\Vert_{\infty}}=:\underline{\eta}>0.
\end{align*}
We deduce that $\phi$ satisfies (\ref{eq:hypocontr}) and  by applying Proposition \ref{prop:asympdet} that  
\begin{equation}\label{proofconv}
\lim_{t\rightarrow\infty}\Vert e^{-\lambda^{*}t}v_t -m_0 N\Vert_{\text{TV}}=0,
\end{equation} 
and that
\begin{equation}\label{proo:convnorm}
\frac{n_t(dx,da)}{\rho(t)}=\frac{e^{-\lambda^{*}t} v_t(x,a)}{e^{-\lambda^{*}t}\int_{\cS\times\R_+}v_t}\overset{\text{TV}}{\underset{t\rightarrow\infty}{\longrightarrow}} N(x,a)dxda.
\end{equation}
We now study the long-time behaviour of $\rho(t)$. Choosing $f=1$ in (\ref{eq:weakformpde})                                                                                                                                                                                                                                                                                                                                                                                                                                                                                                                                                                                                                                                                                                                                                                                                                                                                                                                                                                                                                                                                                                                                                                                                                                                                                                                                                                                                                                                                                                                                                                                                                                                                                                                                                                                                                                                                                                                                                                                                                                                                                                                                                                                                                                                                                                                                                                                                                                                                                                                                                                                                                                                                                                                                                                                                                                                                                                                                                                                                                                                                                                                                                                                                                                                                                                                                                                                                                                                                                                                                                                                                                                                                                                                                                                                                                                                                                                                                                                                                                                                                                                                                                                                                                                                                                                                                                                                                                                                                                                                                                                                                                                                                                                                                                                                                                                                                                                                                                                                                                                                                                                                                                                                                                                                                                                                                                                                                                                                                                                                                                                                                                                                                                                                                                                                                                                                                                                                                                                                                                                                                                                                                                                                                                                                                                                                                                                                                                                                                                                                                                                                                                                                                                                                                                                                                                                                                                                                                                                                                                                                                                                                                                                                                                                                                                                                                                                                                                                                                                                                                                                                                                                                                                                                                                                                                                                                                                                                                                                                                                                                                                                                                                                                                                                                                                                                                                                                                                                                                                                                                                                                                                                                                                                                                                                                                                                                                                                                                                                                                                                                                                                                                                                                                                                                                                                                                                                                                                                                                                                                                                                                                                                                                                                                                                                                                                                                                                                                                                                                                                                                                                                                                                                                                                                                                                                                                                                                                                                                                                                                                                                                                                                                                                                                                                                                                                                                                                                                                                                                                                                                                                                                                                                                                                                                                                                                                                                                                                                                                                                                                                                                                                                                                                                                                                                                                                                                                                                                                                                                                                                                                                                                                                                                                                                                                                                                                                                                                                                                                                                                                                                                                                                                                                                                                                                                                                                                                                                                                                                                                                                                                                                                                                                                                                                                                                                                                                                                                                                                                                                                                                                                                                                                                                                                                                                                                                                                                                                                                                                                                                                                                                                                                                                                                                                                                                                                                                                                                                                                                                                                                                                                                                                                                                                                                                                                                                                                                                                                                                                                                                                                                                                                                                                                                                                                                                                                                                                                                                                                                                                                                                                                                                                                                                                                                                                                                                                                                                                                                                                                                                                                                                                                                                                                                                                                                                                                                                                                                                                                                                                                                                                                                                                                                                                                                                                                                                                                                                                                                                                                                                                                                                                                                                                                                                                                                                                                                                                                                                                                                                                                                                                                                                                                                                                                                                                                                                                                                                                                                                                                                                                                                                                                                                                                                                                                                                                                                                                                                                                                                                                                                                                                                                                                                                                                                                                                                                                                                                                                                                                                                                                                                                                                                                                                                                                                                                                                                                                                                                                                                                                                                                                                                                                                                                                                                                                                                                                                                                                                                                                                                                                                                                                                                                                                                                                                                                                                                                                                                                                                                                                                                                                                                                                                                                                                                                                                                                                                                                                                                                                                                                                                                                                                                                                                                                     it comes that
\begin{align*}
\frac{\text{d}\rho}{\text{d}t}&(t)\\
&=\int_{\cS\times\R_+}n_t(dx,da)\left((1-p)B(x,a)+pB(x,a)\int_{\cS}k(x,a,y)dy-D(x,a)\right) -c\rho^2(t) \\
&=\rho(t)\left( \mathcal{D}(t)+\lambda^{*}\right) -c\rho^2(t)
\end{align*}
with 
\begin{equation*}
\mathcal{D}(t)=\int_{\cS\times\R_+}\frac{n_t(dx,da)}{\rho(t)}\left((1-p)B(x,a)+pB(x,a)\int_{\cS}k(x,a,y)dy-D(x,a)\right) -\lambda^{*}.
\end{equation*}
Using (\ref{proo:convnorm}) we deduce that $\mathcal{D}(t)\rightarrow 0$ as $t\rightarrow\infty$ and that $\rho(t)\rightarrow \frac{\lambda^{*}}{c}$ as $t\rightarrow\infty$ (using a similar method as in \cite{leman2014influence}) that ends the proof.
\end{proof} 
\section{Discussion}
\subsection{Summary of our results and related literature}
We studied the long-time behaviour of an age structured selection-mutation population dynamics. 
The crowding effect only affects the mortality rate of individuals. This is logistical and does not depend on the trait of the affected individual. The renewal term is linear and non-local. The probability of mutation satisfies $p\in\left]0,1\right[$. These assumptions  are similar to those in \cite{bonnefon2015concentration}. Indeed, we extended results of \cite{bonnefon2015concentration} to an age-structured population. 
In terms of age-structured models, \cite{calsina2013steady} shows the existence of stationary states in a function space for a selection mutation dynamics with age structure, in the pure mutation case. In \cite{nordmann2017dynamics}, the authors study issues of concentration in a model similar to ours. In particular, in the pure selection case ($p=0$), they show the convergence of dynamics on traits maximizing fitness. By considering the case $p\in\left]0,1\right[$, our results complement them. We showed the existence of eventually singular stationary solutions. When the stationary measure admits a continuous and bounded density, we obtained global stability (in total variation distance) of  measure solutions of (\ref{eq:edpintro}). 
\\From an application perspective, our assumptions seem difficult to justify. The mortality rate assumption allows us to reduce the study of the non-linear equation to that of the linear problem and to apply a trick similar as in \cite{burger1988perturbations},\cite{metz2014dynamics},\cite{iannelli2017basic}. 
Linear dynamics is studied using arguments similar to those developed in \cite{perthame2006transport} for the age-structured case. We show the existence of principal eigenmeasures for the direct and dual eigenvalue problems. The analysis of the eigenvalue problem is based on a duality approach for studying spectral properties of some non-local operators \cite{coville2010simple},\cite{coville2013singular}.  
\subsection{Extension to more general models}
In \cite{ackleh2016population}, the authors study a differential equation on measures space describing a very general selection mutation dynamics (without age structure). They consider very general birth rates and death rates. Crowding effects affect per capita birth rate, and the effect on per capita mortality depends on trait of affected individual. The authors show boundedness and persistence results for the solutions. Mortality rates of the same type are considered in \cite{desvillettes2008selection} for a pure selection model. It was noted that there is not, in general, a single steady state. The  approach we use in this article does not allow us to obtain similar results in these very general frameworks. Concerning the renewal term, it could be interesting to consider one of the form
\begin{equation*}
\mathcal{F}\left[n_t\right](dx)=\int_{\cS\times\R_+}B(y,a)\gamma(y,a,dx)n_t(dy,da).
\end{equation*}
Note that in this paper, we considered the particularly case 
$\gamma(y,a,dx)=(1-p)\delta_y(dx)+ p k(y,a,x)dx$. Let's try to apply the same method to the more general renewal condition. We have to solve the direct eigenvalue problem
\begin{equation*}
\begin{cases}
-\partial_a N(x,a) - D(x,a)N(x,a) =\lambda N(x,a)\\
N(x,0)=\int_{\cS\times\R_+}B(y,\alpha)\gamma(y,a,dx)N(y,\alpha).
\end{cases}
\end{equation*} 
Following the same ideas as in Section 3, we obtain that the problem is reduced to studying the properties of the operators $\mathcal{F}_\lambda: \mathcal{M}(\cS)\rightarrow \mathcal{M}(\cS)$ defined by
\begin{equation*}
\mathcal{F}_\lambda\left[\mu\right](dx)=\int_{\cS}K_\lambda(dx,y)\mu(dy).
\end{equation*}
Arguing that $\mathcal{F}_\lambda=(G_\lambda)'$ where $G_\lambda:C(\cS)\rightarrow C(\cS)$ is defined by
\begin{equation*}
G_\lambda\left[f\right](x)=\int_{\cS}K_\lambda(dy,x)f(y)
\end{equation*}
we obtain that $\rho(G_\lambda)=\rho(\mathcal{F}_\lambda)$ is an eigenvalue of $\mathcal{F}_\lambda$ associated with a positive eigenmeasure. Then we have to study continuity and monotonicity properties of the map $\lambda\rightarrow\rho(G_\lambda)$. It does not seem possible to deduce such properties in the general case. It should then be considered on a case-by-case basis. 
\\Indeed, when stationary states admits a non-trivial singular part, the topology induced by total variation distance is stronger to obtain global stability of solutions. In the pure selection case ($p=0$), it has been shown in \cite{nordmann2017dynamics} that the trait distribution converges toward Dirac measures. In the case $p\in\left]0,1\right[$, similar arguments can't be used and the problem stays open.

\section*{Acknowledgments} First, I would like to thank Sylvie Méléard for her continual guidance during this work and her many readings of this manuscript. I also want to thank Gael Raoul and Pierre Collet for many interesting discussions. And finally, I would like to thank Jérôme Coville for the discussion we had on the subject. I acknowledge partial support by the Chaire Modélisation Mathématique et Biodiversité of Veolia Environment - \'Ecole Polytechnique - Museum National d'Histoire Naturelle - FX. This work was also supported by a public grant as part of the investissement d'avenir project, reference ANR-11-LABX-0056-LMH, LabEx LMH.

\bibliographystyle{plain} 
\bibliography{bibli} 

\appendix

\section{Operator Theory}
In this appendix, we recall some well-known results about spectral theory for bounded linear operators on Banach space and positive operators.

\subsection{Resolvent and spectrum}
Let $(X,\Vert .\Vert)$ be a complex Banach space. We denote by $\textbf{B}(X)$ the set of all bounded linear maps from $X$ to $X$. Let $T\in\textbf{B}(X)$. We denote by $\Vert T \Vert_{op} =\sup_{\Vert x\Vert =1}\Vert T(x)\Vert$ the operator norm of $T$. 
\begin{defi}Let $T\in \textbf{B}(X)$.
\begin{itemize}
\item[(i)] The resolvent set of $T$ is $\mathcal{R}(T)=\lbrace z\in \mathbb{C}:(z I -T)^{-1}\in \textbf{B}(X) \rbrace$. For all $z\in\mathcal{R}(T)$, the linear map $R_T(z)= (z I - T)^{-1}$ is called resolvent of $T$ at point $z$.
\item[(ii)]The spectrum of $T$ is $\sigma(T)=\mathbb{C}\setminus \mathcal{R}(T)$ and the spectral radius of $T$ is $\rho(T)=\sup\lbrace \vert z\vert:z\in \sigma(T)\rbrace$.
\item[(iii)]The essential spectrum of $T$ is the set $\sigma_e(T)$ of $z\in\sigma(T)$ which satisfy at least one of the following condition: (1) the range of $z I -T$ is not closed; (2) $z$ is not isolated in $\sigma(T)$; (3)  $\cup_{n\geq 1} \ker((z I -T)^n)$ is infinite dimensional. The essential spectral radius of $T$ is $\rho_e(T)=\sup\lbrace \vert z\vert:z\in \sigma_e(T)\rbrace$ 
\end{itemize}
\end{defi}
\begin{rema}
There are different definition of the essential spectrum in the literature, which are not equivalent. The definition we choose has been introduced by Browder \cite{browder1961spectral}. 
\end{rema}
\begin{defi}
Let $z_0$ be a pole of the resolvent. Let
\begin{equation}
R_T(z)=\sum_{k=-m}^{+\infty}a_k(z -z_0)^k\quad
\end{equation}
be the Laurent expansion of $R_T$ near $z_0$ where $a_k$ are linear operators on $X$ and $a_{-m}\neq 0$. The integer $m$ is the order of the pole $z_0$. Then $P(T):= a_{-1}$ is the projector onto the space $\ker((z I -T)^m)$. The dimension of $\ker((z I -T)^m)$ is called algebraic multiplicity of $z_0$. 
\end{defi}
The following characterisation of the essential spectrum is very useful. It is proved in \cite[Lemma 17]{browder1961spectral}.
\begin{prop}\label{prop:browder}
Let $T\in\textbf{B}(X)$ and $z\in\sigma(T)$. Then, $z \notin \sigma_e(T)$ if and only if for some $m\in\N^{*}$, $z$ is a pole of the resolvent of order $m$ such that $\ker((z I -T)^m)$ is finite dimensional. 
\end{prop}
The following result is adapted from Kato \cite[Ch.IV §4. Thm 3.16]{kato2013perturbation}.
\begin{prop}\label{annex:kato}
Let $T\in\textbf{B}(X)$ and $z_0\in \sigma(T)\setminus\sigma_e(T)$. We denote by $\alpha_T(z_0)$ the algebraic multiplicity of $z_0$. Let $\epsilon >0$. There is $\delta > 0$ such that if $\Vert T - S\Vert_{op} <\delta$, the two following assertions are satisfied:
\begin{itemize}
\item[(a)]There is $z\in\sigma(S)\setminus \sigma_e(S)$ such that $\alpha_S(z)=\alpha_T(z_0)$.
\item[(b)]$\Vert P(T)- P(S)\Vert_{op}<\epsilon$.
\end{itemize}
\end{prop}
\subsection{Ordered Banach space}
Let $(X, \Vert .\Vert)$ be a real Banach space. Let $C\subset X$ be a positive cone. We denote always $\textbf{B}(X)$ for the set of bounded linear maps on $X$.
\begin{defi}\label{defi:irred}Let $T\in \textbf{B}(X)$.
\begin{itemize}
\item[(i)]The operator $T$ is positive if $T(C)\subset C$.
\item[(ii)]The operator $T$ is irreducible if $T$ is positive and for some scalar $z > \rho(T)$ and for each non-zero $u\in C$, the element $\sum_{k=1}^{+\infty}z^{-n}T^n(u)$ is quasi interior to $C$ (see \cite{schaefer1971graduate} for the definition of quasi interior point).
\item[(iii)]Let $T,S\in B(X)$ be positive. We denote $T\leq S$ if the bounded linear map $S-T$ is positive.
\end{itemize}
\end{defi}
The following proposition gives some monotonicity properties of the spectral radius. The point (i) is proved in \cite[Theorem 1.1]{burlando1991monotonicity}. The point (ii) is proved in \cite[Theorem 3.9]{gao2013extensions}.
\begin{prop}\label{annex:montoto}Let $S,T\in \textbf{B}(X)$ be positive such that $S\leq T$. We have
\begin{itemize}
\item[(i)]$\rho(S)\leq \rho(T)$;
\item[(ii)]Assume moreover that $\rho(T)$ is a pole of the resolvent of $T$. Then we have either $T=S$ or $\rho(S)<\rho(T)$.
\end{itemize}
\end{prop}
We deduce a result of upper-semi continuity of the spectral radius. It is classical. 
\begin{lemm}\label{lemm:uppersemi}
Let $T\in\textbf{B}(X)$ be positive and let $(T_k)_{k\geq 0}$ be a non-increasing sequence of $\textbf{B}(X)$ such that $\Vert T_k - T\Vert_{\infty}\rightarrow 0$ when $k\rightarrow \infty$. Then $\rho(T_k)\underset{k\rightarrow\infty}{\rightarrow }\rho(T)$.
\end{lemm}
\begin{proof}
By Proposition \ref{annex:montoto} (i), the sequence $\rho(T_k)$ is non-increasing and is bounded below by $\rho(T)$. Then, it converges to a limit $\rho^{*}\geq \rho(T)$. Assume that $\rho^{*}>\rho(T)$. Since the spectral radius is an element of the spectrum, we deduce that for all $k\in\N$ the operator $\rho(T_k)I -T_k$ is singular (e.g $(\rho(T_k)I -T_k)^{-1}$ is not bounded). Moreover, the set of singular operators being closed, we deduce, taking the limit $k\rightarrow\infty$ that $\rho^{*}I -T$ is singular which is absurd since $\rho^{*}>\rho(T)$.
\end{proof}
\subsection{The space \texorpdfstring{$(C(\cS),\Vert .\Vert_{\infty})$}{Lg}}
Let $\cS$ be a compact subset of $\R^d$. We now give some results on the Banach space $(C(\cS), \Vert .\Vert_{\infty})$ where $C(\cS)$ denotes the set of continuous functions from $\cS$ to $\R$ and $\Vert .\Vert_{\infty}$ denotes the uniform norm. We denote by $C^{+}(\cS)$ the cone of non-negative functions on $\cS$. We recall that the space of (signed) Radon measure $\mathcal{M}(\cS)$ is the topological dual (the space of continuous linear form) of $C(\cS)$ and that the set of positive Radon measure $\mathcal{M}^{+}(\cS)$ is the dual cone of $C^{+}(\cS)$. For any $T\in \textbf{B}(C(\cS))$, we denote by $T'\in \textbf{B}(\mathcal{M}(\cS))$ his adjoint. 
\\The next result is easily adapted from \cite[ Appendix §2.2.6]{schaefer1971graduate} (it was originally introduced by Krein-Rutman \cite{krein1948linear}). It is a generalisation of Perron-Frobenius Theorem for positive matrices to the infinite-dimensional framework.
\begin{prop}
Let $T\in \textbf{B}(C(\cS))$ such that $T\geq 0$. Then $\rho(T)$ is an eigenvalue of $T'$ associated with a positive eigenmeasure $\mu\in\mathcal{M}^{+}(\cS)$.
\end{prop}
The following result is proved in \cite[Appendix §3.3.3]{schaefer1971graduate}. It precises the analogy with the Perron-Frobenius theory.
\begin{prop}\label{annex:spectral}
Let $T\in \textbf{B}(C(\cS))$ such that $T$ is irreducible. Then we have:
\begin{itemize}
\item[(i)]The spectral radius $\rho(T)$ is the only possible eigenvalue associated with a non-negative eigenfunction.
\item[(ii)]Assume moreover that $\rho(T)$ is a pole of the resolvent. Then, $\rho(T)$ is a pole of order one with algebraic multiplicity equals to one.
\end{itemize}
\end{prop}
We give a lemma which characterises the quasi interior points (see \cite{schaefer1971graduate} for the definition) of $C^{+}(\cS)$. 
\begin{lemm}\label{lemm:quasi}
Let $f\in C(\cS)$. Then $f$ is quasi interior to $C^{+}(\cS)$ if and only if $f(x)>0$ for all $x\in\cS$.
\end{lemm}
We give now a result on the spectrum of the multiplication operator.
\begin{lemm}\label{lemm:spectremult}
Let $r\in C(\cS)$ be a positive function. Let us denote by $r$ the endomorphism of $C(\cS)$ defined by $rf(x)=r(x)f(x)$. Then we have $\sigma(r)=\sigma_e(r)=\lbrace r(x):x\in\cS\rbrace$.
\end{lemm}
\begin{proof}
We start by proving that $\sigma(r)=\lbrace r(x):x\in\cS\rbrace=:r(\cS)$. Let $y\in r(\cS)^c$. It is straightforward to verify that the operator $f\in C(\cS)\longmapsto \frac{1}{y - r(x)}f$ is the inverse of $(yI - r)$ and is bounded. We deduce that $y\in \mathcal{R}(r)$ and $\sigma(r)\subset r(\cS)$. Conversely, let $y\in \mathcal{R}(r)$ and assume $y=r(x)$ with $x\in\cS$. Let $u:=(yI-r)^{-1}$ and $f\in\cS$ such that $f(x)\neq 0$. We have $f(x)=u(yI-r)f(x)=0$ which is absurd and then $r(\cS)\subset \sigma(r)$. The function $r$ being continuous, the set $\sigma(r)$ has no isolated point and we deduce that $\sigma(r)=\sigma_e(r)$.
\end{proof}

\end{document}